\documentclass{article}
\usepackage{amsmath}
\usepackage{amssymb}
\usepackage{amsthm}
\usepackage{array}
\usepackage{nicematrix}
\usepackage{enumerate}
\usepackage{booktabs}
\usepackage{graphics}
\usepackage{subfigure}
\usepackage[margin=2cm]{geometry}
\usepackage{longtable}
\usepackage[colorlinks,linkcolor=blue]{hyperref}
\usepackage{indentfirst}

\newtheorem{definition}{Definition}
\newtheorem{theorem}{Theorem}

\newtheorem{lemma}{Lemma}
\newtheorem{proportion}{Proportion}
\newtheorem{corollary}{Corollary}
\newtheorem{assumption}{Assumption}
\newtheorem{algorithm}{Algorithm}

\title{A Dual Method for Minimax Quadratic Programming
\author{Wenhui Ren\thanks{\baselineskip 9pt School of Mathematical Sciences, Dalian University of Technology, Dalian 116024, China. E-mail: renwhmath@mail.dlut.edu.cn} \
 and \ Liwei Zhang\thanks{\baselineskip 9pt
		School of Mathematical Sciences, Dalian University of Technology, Dalian 116024, China. E-mail: lwzhang@dlut.edu.cn. The
		research of this author was supported by the National Key R\&D Program of China under project No. 2022YFA1004000 and the National Natural Science Foundation of China under project  No. 12371298.}}
}

\begin{document}
	\maketitle
	\begin{abstract}
		This paper investigates minimax quadratic programming problems with coupled inequality constraints. By leveraging a duality theorem, we develop a dual algorithm that extends the dual active set method to the minimax setting, transforming the original inequality constrained problem into a sequence of equality constrained subproblems. 
		Under a suitable assumption, we prove that the associated S-pairs do not repeat and that the algorithm terminates in a finite number of iterations, guaranteed by the monotonic decrease of the objective function value. To ensure numerical stability and efficiency, the algorithm is implemented using Cholesky factorization and Givens rotations. Numerical experiments on both randomly generated minimax quadratic programs and illustrative applications demonstrate the accuracy, stability, and computational effectiveness of the proposed algorithm.
	\end{abstract}
	
	

	\baselineskip 18pt
	\parskip 2pt
	
	\section{Introduction}
	Consider the following minimax quadratic optimization problem:
	\begin{align}\label{a1.1}
		\left\lbrace
		\begin{aligned}
			\min_{x\in\mathbb{R}^{n_{x}}}\max_{y\in\mathbb{R}^{n_{y}}}\quad& f(x, y) = \frac{1}{2}x^{T}G_{11}x + x^{T}G_{12}y + \frac{1}{2}y^{T}G_{22}y + c_{x}^{T}x + c_{y}^{T}y,\\
			{\rm s.t.}\quad & s(x, y) = Ax + By + h \leq 0.
		\end{aligned}\right.
	\end{align}
	where $G_{11} \in \mathbb{R}^{n_{x} \times n_{x}}$ is a symmetric matrix, $G_{22} \in \mathbb{R}^{n_{y} \times n_{y}}$ is a symmetric negative definite matrix, $G_{12} \in \mathbb{R}^{n_{x} \times n_{y}}$, $c_{x} \in \mathbb{R}^{n_{x}}$, $c_{y} \in \mathbb{R}^{n_{y}}$, $A \in \mathbb{R}^{m \times n_{x}}$, $B \in \mathbb{R}^{m \times n_{y}}$ and $h \in \mathbb{R}^{m}$. Although the vectors of variables $x$ and $y$ may also be subject to equality constraints
	\begin{align*}
		l(x, y) = A^{l}x + B^{l}y + h^{l} = 0,
	\end{align*}
	we will ignore such constraints for the moment in order to simplify our presentation.
	
	Minimax optimization problems arise in a broad range of fields, from modern machine learning such as generative adversarial networks, adversarial training, and multi-agent reinforcement learning to classical areas including saddle point problems, numerical partial differential equations, and game theory. Extensive studies have focused on unconstrained minimax problems and those without coupling constraints between the inner variable and the outer variable~\cite{nesterovDualExtrapolationIts2007, linGradientDescentAscent2020, mokhtariConvergenceRate$mathcalO12020, zhangPrimalDualFirstOrderMethods2023}. More recently, several works have investigated minimax optimization problems with coupled equality constraints~\cite{daiOptimalityConditionsNumerical2024, daiRateConvergenceAugmented2024}. For general nonlinear programming, the sequential quadratic programming (SQP) method represents a cornerstone, providing an efficient and reliable approach for solving constrained problems~\cite{boggsSequentialQuadraticProgramming1995, gillSequentialQuadraticProgramming2012}. We conjecture that extending the SQP method to minimax problems requires solving a quadratic subproblem at each iteration to obtain the search direction. Consequently, the development of an efficient algorithm for minimax quadratic programming with inequality constraints is of particular interest.
	
	Quadratic programming (QP) has a long history and numerous applications, giving rise to a variety of algorithmic methods. Representative approaches include the primal methods of Beale~\cite{bealeMinimizingConvexFunction1955}, Wolfe~\cite{wolfeSimplexMethodQuadratic1959}, and Fletcher~\cite{fletcherGeneralQuadraticProgramming1971}; the dual methods of Lemke~\cite{lemkeMethodSolutionQuadratic1962}, Van de Panne and Whinston~\cite{vandepanneSimplexDualMethod1964}, and Goldfarb and Idnani~\cite{goldfarbNumericallyStableDual1983}; the pivoting method of Keller~\cite{kellerGeneralQuadraticOptimization1973}; and the iterative schemes of Hildreth~\cite{hildrethQuadraticProgrammingProcedure1957}, Frank and Wolfe~\cite{frankAlgorithmQuadraticProgramming1956}, Pang~\cite{pangHybridMethodSolution1981}, Herman and Lent~\cite{hermanFamilyIterativeQuadratic1978}, and Mangasarian~\cite{mangasarianSolutionSymmetricLinear1977}. Among these, the Goldfarb and Idnani method~\cite{goldfarbNumericallyStableDual1983} is one of the most influential dual active set algorithms for strictly convex QPs. Powell~\cite{powellQuadraticProgrammingAlgorithm1985} implemented this method and demonstrated its robustness and efficiency through extensive numerical experiments. Later, Boland~\cite{bolandDualactivesetAlgorithmPositive1996} extended the approach to positive semi-definite problems. Motivated by these developments, we aim to propose a dual active set method to minimax quadratic programming with coupled inequality constraints, following the key ideas of the Goldfarb and Idnani method.
	
	The remainder of this paper is organized as follows.
	In Section~\ref{section2}, we recall the notion of a local minimax point for Problem~\eqref{a1.1} and the associated optimality conditions proposed by Dai and Zhang~\cite{daiOptimalityConditionsConstrained2020}.
	In Section~\ref{section3}, we introduce the basic framework for solving minimax quadratic programs with inequality constraints, inspired by the Goldfarb and Idnani method.
	In Section~\ref{section4}, we develop a dual algorithm for minimax quadratic programming and establish, under a mild assumption, that the algorithm terminates finitely by demonstrating monotonic descent of the objective value.
	In Section~\ref{section5}, we employ Cholesky decomposition and Givens rotations to obtain an efficient and numerically stable implementation.
	In Section~\ref{section6}, we present numerical experiments to demonstrate the performance of the proposed algorithm, including randomly generated minimax quadratic programs, an adversarial attack on a mean-covariance portfolio model, and illustrative examples highlighting the algorithm's internal iterations.
	Finally, Section~\ref{section7} concludes this paper.
	
	\section{Optimality conditions for constrained minimax problem}\label{section2}
	Considering the following problem:
	\begin{align}\label{a2.1}
		\left\lbrace\begin{aligned}
			\min_{x\in\mathbb{R}^{n_{x}}}\max_{y\in\mathbb{R}^{n_{y}}}\quad& f(x,y),\\
			{\rm s.t.}\quad& l(x, y) = 0,\\
			& s(x, y) \leq 0.
		\end{aligned}\right.
	\end{align}
	where $n_{x}, n_{y}, m_{l}$ and $m_{s}$ are positive integers and $f : \mathbb{R}^{n_{x}} \times \mathbb{R}^{n_{y}} \rightarrow \mathbb{R}$, $l : \mathbb{R}^{n_{x}} \times \mathbb{R}^{n_{y}} \rightarrow \mathbb{R}^{m_{l}}$ and $s : \mathbb{R}^{n_{x}} \times \mathbb{R}^{n_{y}} \rightarrow \mathbb{R}^{m_{s}}$ are twice continuously differentiable in a neighborhood of some feasible point $(x^{\ast}, y^{\ast}) \in \mathbb{R}^{n_{x}} \times \mathbb{R}^{n_{y}}$. Problem~\eqref{a2.1} can be rewritten as the following equivalent form,
	\begin{align}\label{a2.2}
		\min_{x\in\mathbb{R}^{n_{x}}}\max_{y \in Y(x)} f(x,y),
	\end{align}
	where
	\begin{align}\label{definition of Y(x)}
		Y(x) = \left\lbrace y \in\mathbb{R}^{n_{y}} : l(x, y) = 0, s(x, y) \leq 0\right\rbrace.
	\end{align}
	For unconstrained continuous minimax optimization; namely the problem with $Y(x) = \mathbb{R}^{n_{y}}$, Jin et al.~\cite{jinWhatLocalOptimality2020} proposed a proper definition of local optimality called local minimax. Dai and Zhang~\cite{daiOptimalityConditionsConstrained2020} extended this definition of local minimax point for constrained minimax problem including~\eqref{a2.2} as a special case. Now we give the definition of the local minimax point for Problem~\eqref{a2.2}.
	\begin{definition}\label{minimax point}
		A point $(x^{\ast}, y^{\ast}) \in \mathbb{R}^{n_{x}} \times \mathbb{R}^{n_{y}}$ is said to be a local minimax point of Problem~\eqref{a2.2} if there exist $\delta_{0} > 0$ and a function $\eta : \left( 0, \delta_{0}\right] \rightarrow \mathbb{R}_{+}$ satisfying $\eta(\delta) \rightarrow 0$ as $\delta \rightarrow 0$ such that for any $\delta \in \left( 0, \delta_{0}\right]$ and any $(x, y) \in \mathbb{B}_{\delta}(x^{\ast}) \times\left[Y(x^{\ast}) \cap\mathbb{B}_{\delta}(y^{\ast})\right]$, we have
		\begin{align}
			f(x^{\ast}, y) \leq f(x^{\ast}, y^{\ast}) \leq \max_{\omega}\left\lbrace f(x, \omega) : \omega \in Y(x)\cap\mathbb{B}_{\eta(\delta)}(y^{\ast})\right\rbrace.
		\end{align}
	\end{definition}
	We can understand Definition~\ref{minimax point} in this way: $y^{\ast}$ is the local maximal point of $f(x^{\ast}, \cdot)$, while $x^{\ast}$ is the local minimal point of a surrogate objective function $\max_{y}\left\lbrace f(x, y): \|y - y^{\ast}\|\leq \eta(\delta), l(x, y) = 0, s(x, y) \leq 0\right\rbrace$, which can be viewed as a local approximation of the objective function $\max_{y}\left\lbrace f(x, y) : l(x, y) = 0, s(x, y) \leq 0\right\rbrace$.
	
	For a point $x \in \mathbb{R}^{n_{x}}$ around $x^{\ast}$, we use $(P_{x})$ to denote the following problem 
	\begin{align*}
		\begin{aligned}
			\max_{\omega \in \mathbb{R}^{n_{y}}}\quad & f(x, \omega)\\
			{\rm s.t.}\quad & l(x, \omega) = 0,\\
			& s(x, \omega) \leq 0.
		\end{aligned}
	\end{align*}
	The Lagrangian of Problem $(P_{x})$ is defined by
	\begin{align*}
		\mathcal{L}(x; \omega, \mu, \lambda) = f(x, \omega) + \mu^{T}l(x, \omega) + \lambda^{T}s(x, \omega).
	\end{align*}
	In order to state the second order optimality conditions for Problem~\eqref{a2.2}, we require the following conditions.
	\begin{definition}
		Let $(\mu^{\ast}, \lambda^{\ast}) \in \mathbb{R}^{m_{l}} \times \mathbb{R}^{m_{s}}$ be a point. We say that Jacobian uniqueness conditions of Problem $(P_{x^{\ast}})$ are satisfied at $(y^{\ast}, \mu^{\ast}, \lambda^{\ast})$ if
		\begin{enumerate}
			\item The point $(y^{\ast}, \mu^{\ast}, \lambda^{\ast})$ is a Karush-Kuhn-Tucker point of Problem $(P_{x^{\ast}})$; namely,
			\begin{align*}
				& \nabla_{y} \mathcal{L}(x^{\ast}; y^{\ast}, \mu^{\ast}, \lambda^{\ast}) = 0,\\
				& l(x^{\ast}, y^{\ast}) = 0,\\
				& 0 \geq \lambda^{\ast} \perp s(x^{\ast}, y^{\ast}) \leq 0.
			\end{align*}
			\item The linear independence constraint qualification holds at $y^{\ast}$; namely, the set of vectors
			\begin{align*}
				\left\lbrace \nabla_{y}l_{1}(x^{\ast}, y^{\ast}), \cdots, \nabla_{y}l_{m_{l}}(x^{\ast}, y^{\ast})\right\rbrace \cup \left\lbrace\nabla_{y}s_{i}(x^{\ast}, y^{\ast}) : i\in I_{x^{\ast}}(y^{\ast})\right\rbrace
			\end{align*}
			are linearly independent, where $I(x^{\ast}, y^{\ast}) = \left\lbrace i \mid s_{i}(x^{\ast}, y^{\ast}) = 0, i = 1, 2, \cdots, m_{s}\right\rbrace$.
			\item The strict complementarity condition holds at $y^{\ast}$ for $\lambda^{\ast}$; namely,
			\begin{align*}
				\lambda_{i}^{\ast} + s_{i}(x^{\ast}, y^{\ast}) < 0, i = 1, \cdots, m_{s}. 
			\end{align*}
			\item The second-order sufficient optimality condition holds at $(y^{\ast}, \mu^{\ast}, \lambda^{\ast})$,
			\begin{align*}
				\left\langle \nabla_{yy}^{2}\mathcal{L}(x^{\ast}; y^{\ast}, \mu^{\ast}, \lambda^{\ast})d_{y}, d_{y}\right\rangle < 0 \quad \forall d_{y}\in \mathcal{C}_{x^{\ast}}(y^{\ast}),
			\end{align*}
			where $\mathcal{C}_{x^{\ast}}(y^{\ast})$ is the critical cone of Problem $(P_{x^{\ast}})$ at $y^{\ast}$,
			\begin{align*}
				\mathcal{C}_{x^{\ast}}(y^{\ast}) = \left\lbrace d_{y} \in \mathbb{R}^{n_{y}} : \mathcal{J}_{y}l(x^{\ast}, y^{\ast})d_{y} = 0; \nabla_{y}s_{i}(x^{\ast}, y^{\ast})d_{y} \leq 0, i\in I_{x^{\ast}}(y^{\ast}); \nabla_{y} f(x^{\ast}, y^{\ast})d_{y} \leq 0\right\rbrace.
			\end{align*}
		\end{enumerate}
	\end{definition}
	
	Under the Jacobian uniqueness conditions, we can easily obtain the following result.
	\begin{lemma}\label{lemma the suf}
		Let $(x^{\ast}, y^{\ast}) \in \mathbb{R}^{n_{x}}\times\mathbb{R}^{n_{y}}$ be a point around which $f$, $l$ and $s$ are twice continuously differentiable. Let $(\mu^{\ast}, \lambda^{\ast}) \in \mathbb{R}^{n_{x}}\times\mathbb{R}^{n_{y}}$ such that Jacobian uniqueness conditions of Problem $(P_{x^{\ast}})$ are satisfied at $(y^{\ast}, \mu^{\ast}, \lambda^{\ast})$. Then there exist $\delta_{0} > 0$ and $\varepsilon_{0} > 0$, and a twice continuously differentiable mapping $(y, \mu, \lambda) : \mathbb{B}_{\delta_{0}}(x^{\ast}) \rightarrow \mathbb{B}_{\varepsilon_{0}}(y^{\ast})\times\mathbb{B}_{\varepsilon_{0}}(\mu^{\ast})\times\mathbb{B}_{\varepsilon_{0}}(\lambda^{\ast})$ such that Jacobian uniqueness conditions of Problem $(P_{x})$ are satisfied at $(y(x), \mu(x), \lambda(x))$ when $x\in\mathbb{B}_{\varepsilon_{0}}(x^{\ast})$.
	\end{lemma}
	
	To state the second-order sufficiency optimality conditions for Problem~\eqref{a2.2}, we define
	\begin{align}\label{KRM}
		\begin{aligned}
			& K(x, y, \mu, \lambda) = \begin{bmatrix}
				\nabla_{yy}^{2} \mathcal{L}(x; y, \mu, \lambda) & -\mathcal{J}_{y}l(x, y)^{T} & -\mathcal{J}_{y}s(x, y)^{T}\\
				-\mathcal{J}_{y}l(x, y) & 0 & 0\\
				-\mathcal{J}\Pi_{\mathbb{R}_{+}^{m_{s}}}(s(x, y) + \lambda)\mathcal{J}_{y}s(x, y) & 0 & I_{m_{s}} - \mathcal{J}\Pi_{\mathbb{R}_{+}^{m_{s}}}(s(x, y) + \lambda)
			\end{bmatrix},\\
			& R(x, y, \mu, \lambda) = \begin{bmatrix}
				\nabla_{xy}^{2}\mathcal{L}(x; y, \mu, \lambda) & -\mathcal{J}_{x}l(x, y)^{T} & -\mathcal{J}_{x}s(x, y)^{T}
			\end{bmatrix},\\
			& M(x, y, \mu, \lambda) = \begin{bmatrix}
				\nabla_{yx}^{T}\mathcal{L}(x; y, \mu, \lambda)\\
				-\mathcal{J}_{x}l(x, y)\\
				-\mathcal{J}\Pi_{\mathbb{R}_{+}^{m_{s}}}(s(x, y) + \lambda)\mathcal{J}_{x}s(x, y)
			\end{bmatrix}.
		\end{aligned}
	\end{align}
	
	\begin{theorem}[Second-order Sufficient Optimality Conditions]\label{Second-order theorem}
		Let $(x^{\ast}, y^{\ast}) \in \mathbb{R}^{n_{x}} \times \mathbb{R}^{n_{y}}$ be a point around which $f$, $l$ and $s$ are twice continuously differentiable and $y^{\ast} \in Y(x^{\ast})$. Let $(\mu^{\ast}, \lambda^{\ast}) \in \mathbb{R}^{m_{l}} \times \mathbb{R}^{m_{s}}$ and assume that Problem $(P_{x^{\ast}})$ satisfies the Jacobian uniqueness conditions at $(y^{\ast}, \mu^{\ast}, \lambda^{\ast})$. Suppose that
		\begin{align*}
			\nabla_{x}\mathcal{L}(x^{\ast}; y^{\ast}, \mu^{\ast}, \lambda^{\ast}) = 0
		\end{align*}
		and
		\begin{align*}
			\nabla_{xx}^{2}\mathcal{L}(x^{\ast}; y^{\ast}, \mu^{\ast}, \lambda^{\ast}) - R(x^{\ast})K(x^{\ast})^{-1}M(x^{\ast}) \succ 0,
		\end{align*}
		where $K(x^{\ast})$, $R(x^{\ast})$ and $M(x^{\ast})$ are defined by \eqref{KRM}. Then there exist $\delta_{1} \in (0, \delta_{0})$, $\varepsilon_{1}\in(0, \varepsilon_{0})$ (where $\delta_{0}$ and $\varepsilon_{0}$ are given by Lemma~\ref{lemma the suf}) and $\gamma_{1} > 0$, $\gamma_{2} > 0$ such that for $x\in\mathbb{B}_{\delta_{1}}(x^{\ast})$ and $y \in \mathbb{B}_{\varepsilon_{1}}(y^{\ast})\cap Y(x^{\ast})$,
		\begin{align}
			f(x^{\ast}, y) + \frac{\gamma_{1}}{2}\|y - y^{\ast}\|^{2} \leq f(x^{\ast}, y^{\ast}) \leq \sup_{\omega \in Y(x)\cap\mathbb{B}_{\varepsilon_{0}}(y^{\ast})} f(x, \omega) - \frac{\gamma_{2}}{2}\|x - x^{\ast}\|^{2},
		\end{align}
		which indicates that $(x^{\ast}, y^{\ast})$ is a local minimax point of Problem~\eqref{a2.2}.
	\end{theorem}
	
	\begin{theorem}[Necessary Optimality Conditions]\label{second-order theorem necessary}
		Let $(x^{\ast}, y^{\ast}) \in \mathbb{R}^{n_{x}} \times \mathbb{R}^{n_{y}}$ be a point around which $f$, $l$ and $s$ are twice continuously differentiable and $y^{\ast} \in Y(x^{\ast})$. Let $(x^{\ast}, y^{\ast})$ be a local minimax point of Problem~\eqref{a2.2}. Assume that the linear independence constraint qualification holds at $y^{\ast}$ for constraint set $Y(x^{\ast})$. Then there exists a unique vector $(\mu^{\ast}, \lambda^{\ast}) \in \mathbb{R}^{m_{l}} \times \mathbb{R}^{m_{s}}$ such that
		\begin{align*}
			& \nabla_{y} \mathcal{L}(x^{\ast}; y^{\ast}, \mu^{\ast}, \lambda^{\ast}) = 0,\\
			& l(x^{\ast}, y^{\ast}) = 0,\\
			& 0 \geq \lambda^{\ast} \perp s(x^{\ast}, y^{\ast}) \leq 0.
		\end{align*}
		For any $d_{y} \in \mathcal{C}_{x^{\ast}}(y^{\ast})$, we have that
		\begin{align*}
			\left\langle \nabla_{yy}^{2}\mathcal{L}(x^{\ast}; y^{\ast}, \mu^{\ast}, \lambda^{\ast})d_{y}, d_{y}\right\rangle \leq 0.
		\end{align*}
		Assuming Problem $(P_{x^{\ast}})$ satisfies Jacobian uniqueness conditions at $(y^{\ast}, \mu^{\ast}, \lambda^{\ast})$, then
		\begin{align*}
			\nabla_{x} \mathcal{L}(x^{\ast}; y^{\ast}, \mu^{\ast}, \lambda^{\ast}) = 0,
		\end{align*}
		and for every $d_{y}\in \mathcal{C}(x^{\ast})$,
		\begin{align*}
			\left\langle \left[\nabla_{xx}^{2}\mathcal{L}(x^{\ast}; y^{\ast}, \mu^{\ast}, \lambda^{\ast}) - R(x^{\ast})K(x^{\ast})^{-1}M(x^{\ast})\right]d_{x}, d_{x}\right\rangle \geq 0,
		\end{align*}
		where $\mathcal{C}(x^{\ast}) = \left\lbrace d_{x}\in\mathbb{R}^{n_{x}} : \nabla_{x} \mathcal{L}(x^{\ast}; y^{\ast}, \mu^{\ast}, \lambda^{\ast})^{T}d_{x} \leq 0\right\rbrace = \mathbb{R}^{n_{x}}$.
	\end{theorem}
	
	In subsequent section, we focus on solving the minimax quadratic optimization problem with linear inequalities \eqref{a1.1}. Using $z$ to denote $(x, y)$, Problem~\eqref{a1.1} can be rewritten as follows:
	\begin{align}\label{realPro}
		\left\lbrace
		\begin{aligned}
			\min_{x\in\mathbb{R}^{n_{x}}}\max_{y\in\mathbb{R}^{n_{y}}}\quad& f(z) = \frac{1}{2}z^{T}Gz + c^{T}z,\\
			{\rm s.t.}\quad & s(z) = Dz + h \leq 0.
		\end{aligned}\right.
	\end{align}
	where $G = \begin{bmatrix}
		G_{11} & G_{12}\\
		G_{12}^{T} & G_{22}
	\end{bmatrix}$, $c = \begin{bmatrix}
		c_{x}\\
		c_{y}
	\end{bmatrix}$, and $D = \begin{bmatrix}
		A & B
	\end{bmatrix}$.
	
	\begin{corollary}[Second-order Sufficient Optimality Conditions For Problem~\eqref{realPro}]\label{corollary}
		Let $z^{\ast} \in \mathbb{R}^{n_{x}}\times\mathbb{R}^{n_{y}}$ be a point around $y^{\ast} \in Y(x^{\ast})$ which is defined by \eqref{definition of Y(x)}. Let $\lambda^{\ast} \in \mathbb{R}^{m}$ and suppose that
		\begin{align*}
			\left\lbrace\begin{aligned}
				& Gz^{\ast} + D^{T}\lambda^{\ast} = 0,\\
				& 0 \geq \lambda^{\ast} \perp Dz^{\ast} + h \leq 0,\\
				& \lambda^{\ast} + Dz^{\ast} + h < 0,\\
				& G_{11} - \begin{bmatrix}
					G_{12} & A_{I_{x^{\ast}}(y^{\ast})}^{T}
				\end{bmatrix} \begin{bmatrix}
					G_{22} & B_{I_{x^{\ast}}(y^{\ast})}^{T}\\
					B_{I_{x^{\ast}}(y^{\ast})} & 0
				\end{bmatrix}^{-1} \begin{bmatrix}
					G_{12}^{T}\\
					A_{I_{x^{\ast}}(y^{\ast})}
				\end{bmatrix} \succ 0,
			\end{aligned}\right.
		\end{align*}
		and 
		\begin{align*}
			\left\lbrace B_{i} : i\in I_{x^{\ast}}(y^{\ast})\right\rbrace
		\end{align*}
		are linearly independent where $I_{x^{\ast}}(y^{\ast}) = \left\lbrace i \mid D_{i}z^{\ast} + h_{i} = 0, i = 1, 2, \cdots, m\right\rbrace$. Then there exist $\delta_{1} \in (0, \delta_{0})$, $\varepsilon_{1}\in(0, \varepsilon_{0})$ (where $\delta_{0}$ and $\varepsilon_{0}$ are given by Lemma~\ref{lemma the suf}) and $\gamma_{1} > 0$, $\gamma_{2} > 0$ such that for $x\in\mathbb{B}_{\delta_{1}}(x^{\ast})$ and $y \in \mathbb{B}_{\varepsilon_{1}}(y^{\ast})\cap Y(x^{\ast})$,
		\begin{align*}
			f(x^{\ast}, y) + \frac{\gamma_{1}}{2}\|y - y^{\ast}\|^{2} \leq f(x^{\ast}, y^{\ast}) \leq \sup_{\omega \in Y(x)\cap\mathbb{B}_{\varepsilon_{0}}(y^{\ast})} f(x, \omega) - \frac{\gamma_{2}}{2}\|x - x^{\ast}\|^{2},
		\end{align*}
		which indicates that $(x^{\ast}, y^{\ast})$ is a local minimax point of Problem~\eqref{realPro}.
	\end{corollary}

	\section{Basic method}\label{section3}
	The dual algorithm presented in the next section belongs to the class of \emph{active set methods}. 
	An \emph{active set} refers to a subset of the $m$ constraints in Problem~\eqref{realPro} that are satisfied as equalities at the current estimate $z$ of the solution to Problem~\eqref{realPro}. 
	We denote by $[m] = \{1, 2, \ldots, m\}$ the index set of all constraints, and by $\alpha \subseteq [m]$ the index set corresponding to the currently active constraints.
	
 	\begin{definition}
		Let $J \subseteq [m]$ be a set of indices. 
		A problem is said to be a \emph{subproblem} $P(J)$ of Problem~\eqref{realPro} if it takes the following form:
		\begin{align}
			P(J)\left\lbrace\begin{aligned}
				\min_{x}\max_{y}\quad& f(z) = \frac{1}{2}z^{T}Gz + c^{T}z,\\
				{\rm s.t.}\quad& s_{J}(z) = D_{J}z + h_{J} \leq 0.
			\end{aligned}\right.
		\end{align}
	\end{definition}
	
	Consider the problem $P([m])$. 
	Suppose that $P([m])$ admits a local minimax point satisfying Corollary~\ref{corollary}, denoted by $z^{\ast}$. 
	We then define
	\begin{align*}
		\alpha^{\ast} = \left\lbrace i\in [m] \mid D_{i}z^{\ast} + h_{i} = 0\right\rbrace.
	\end{align*}
	Then $z^{\ast}$ is a local minimax point of the following problem subject to equality constraints:
	\begin{align}\label{equality}
		\begin{aligned}
			\min_{x\in\mathbb{R}^{n_{x}}}\max_{y\in\mathbb{R}^{n_{y}}}\quad& \frac{1}{2}z^{T}Gz + c^{T}z,\\
			{\rm s.t.}\quad & D_{\alpha^{\ast}}z + h_{\alpha^{\ast}} = 0.
		\end{aligned}
	\end{align}
	and there exists $\lambda^{\ast}\in\mathbb{R}^{m}$ satisfying 
	\begin{align}\label{S-pair1}
		\left\lbrace\begin{aligned}
			& Gz^{\ast} + c + \sum_{i \in \alpha^{\ast}}\lambda_{i}^{\ast}D_{i}^{T} = 0,\\
			& D_{\alpha^{\ast}}z^{\ast} + h_{\alpha^{\ast}} = 0,\\
			& \lambda_{i}^{\ast} \leq 0, i \in \alpha^{\ast},\\
			& \lambda_{j}^{\ast} = 0, i \in [m]\backslash\alpha^{\ast},
		\end{aligned}\right.
	\end{align}
	and
	\begin{align}\label{S-pair2}
		G_{11} - \begin{bmatrix}
			G_{12} & A_{\alpha^{\ast}}^{T}
		\end{bmatrix} \begin{bmatrix}
		G_{22} & B_{\alpha^{\ast}}^{T}\\
		B_{\alpha^{\ast}} & 0
		\end{bmatrix}^{-1} \begin{bmatrix}
		G_{12}^{T}\\
		A_{\alpha^{\ast}}
		\end{bmatrix} \succ 0.
	\end{align}
	For convenient, we define
	\begin{align}
		\Gamma_{\alpha} = G_{11} - \begin{bmatrix}
			G_{12} & A_{\alpha}^{T}
		\end{bmatrix} \begin{bmatrix}
			G_{22} & B_{\alpha}^{T}\\
			B_{\alpha} & 0
		\end{bmatrix}^{-1} \begin{bmatrix}
			G_{12}^{T}\\
			A_{\alpha}
		\end{bmatrix}.
	\end{align}
	
	We have that $z^{\ast}$ is a local minimax point of both Problem~$P([m])$ and Problem~\eqref{equality}. 
	This observation implies that solving Problem~\eqref{equality} yields a local minimax point of $P([m])$. 
	Dai and Zhang~\cite{daiRateConvergenceAugmented2024} proposed an effective algorithm for solving minimax problems with equality constraints. 
	By employing the active set idea, an inequality constrained problem can be transformed into an equality constrained one. 
	To characterize the pair $(z^{\ast}, \alpha^{\ast})$, we introduce the following definition.
	
	\begin{definition}\label{definition of S-pair}
		A pair $(z, \alpha)$, consisting of a point $z$ and an index set $\alpha \subseteq J$, 
		is said to be an \emph{S-solution pair} of the subproblem $P(J)$ if and only if 
		$z$ is an optimal solution of $P(J)$ satisfying \eqref{S-pair1} and \eqref{S-pair2}, 
		$\alpha = \{\, i \in J \mid D_i z + h_i = 0 \,\}$, 
		and the set $\{\, D_i \mid i \in \alpha \,\}$ is linearly independent.
	\end{definition}
	
	Each S-pair $(z, \alpha)$ corresponds to a local minimax point of a subproblem $P(J)$ satisfying $D_{J}z + h_{J} \leq 0$. 
	This observation indicates that our goal is to find an S-pair $(z^{\ast}, \alpha^{\ast})$ associated with the subproblem $P([m])$. 
	We therefore outline below the basic framework for solving Problem~\eqref{realPro}.
	
	\begin{algorithm}{Basic method:}\label{Basic Method}
		\begin{enumerate}
			\item[Step 0] Assume that some S-pair $(z, \alpha)$ is given.
			\item[Step 1] Repeat until all constraints are satisfied:
			\begin{enumerate}
				\item Choose a violated constraint $p \in [m] \backslash\alpha$.
				\item If $P(\alpha\cup\left\lbrace p\right\rbrace)$ is infeasible. STOP-Problem is infeasible.
				\item Else, obtain a new S-pair $(\bar{z}, \bar{\alpha}\cup\left\lbrace p\right\rbrace)$ where $\bar{\alpha} \subseteq \alpha$ and $f(\bar{z}) < f(z)$ and set $(z, \alpha) \leftarrow (\bar{z}, \bar{\alpha}\cup\left\lbrace p\right\rbrace)$.
			\end{enumerate}
			\item[Step 2] STOP-$z$ is the optimal solution for Problem.
		\end{enumerate}
	\end{algorithm}
	
	For the unconstrained minimax subproblem $P(\emptyset)$, the solution 
	$z_{0} = -G^{-1}c$ is available, and thus Algorithm~\ref{Basic Method} can always be initialized with the S-pair $(z_{0}, \emptyset)$.
	
	In the next section, we present a dual algorithm that implements the basic method described above. 
	In Step~1(c), a new S-pair $(\bar{z}, \bar{\alpha} \cup \{ p \})$ is obtained, where we note that $\bar{z}$ may not be a local minimax point of the subproblem $P(\alpha \cup \{ p \})$.
	
	To describe the algorithm precisely, we first introduce some notation. 
	The matrix of normal vectors associated with the active constraints indexed by $\alpha$ is denoted by $N_{\alpha}$, that is, $N_{\alpha} = D_{\alpha}^{T}$. 
	We denote by $\alpha_{+}$ the set $\alpha \cup \{ p \}$, where $p \in [m] \setminus \alpha$, 
	and by $\alpha_{-}$ the set $\alpha \setminus \{ k \}$, where $k \in \alpha$. 
	For convenience, we always use $p$ to denote the index of the constraint to be added to the active set $\alpha$, 
	and $k$ to denote the index of the constraint to be dropped from $\alpha$. 
	We use $n^{+}$ to denote the normal vector $n_{p}$ added to $N_{\alpha}$ to form $N_{\alpha_{+}}$, 
	and $n^{-}$ to denote the column removed from $N_{\alpha}$ to form $N_{\alpha_{-}}$. 
	The symbol $I$ denotes the identity matrix, and $e_{j}$ the $j$th column of $I$.
	
	When the columns of $N_{\alpha}$ are linearly independent, we define the operators
	\begin{align}\label{definition of Nstar}
		N_{\alpha}^{\ast} &= (N_{\alpha}^{T}G^{-1}N_{\alpha})^{-1}N_{\alpha}^{T}G^{-1},
	\end{align}
	and
	\begin{align}\label{definition of H}
		H_{\alpha} &= G^{-1}(I - N_{\alpha}N_{\alpha}^{\ast}) 
		= G^{-1} - G^{-1}N_{\alpha}(N_{\alpha}^{T}G^{-1}N_{\alpha})^{-1}N_{\alpha}^{T}G^{-1}.
	\end{align}
	Here, $N_{\alpha}^{\ast}$ is the so-called pseudo-inverse (or Moore-Penrose generalized inverse) of $N_{\alpha}$. 
	Using $H_{\alpha}$ and $N_{\alpha}^{\ast}$, we next establish several conditions for identifying a local minimax point, 
	which differ from those given in Corollary~\ref{corollary}.
	
	\begin{lemma}\label{lemma equality}
		Consider the problem:
		\begin{align}\label{Section3-1}
			\left\lbrace
			\begin{aligned}
				\min_{x \in \mathbb{R}^{n_x}} \max_{y \in \mathbb{R}^{n_y}} \quad & f(z) = \tfrac{1}{2} z^{T} G z + c^{T} z, \\
				\text{s.t.} \quad & s(z) = N_{\alpha}^{T} z + h_{\alpha} = 0,
			\end{aligned}
			\right.
		\end{align}
		where $B_{\alpha}$ is a row full-rank matrix, $G_{22}$ is negative definite, and $\Gamma_{\alpha}$ is positive definite.  
		If $\hat{z} \in \mathcal{M}_{\alpha} := \{ z \mid N_{\alpha}^{T} z + h_{\alpha} = 0 \}$, then the local minimax point is attained at
		\begin{align*}
		\bar{z} = \hat{z} - H_{\alpha} g(\hat{z}),
		\end{align*}
		where $g(\hat{z}) = G \hat{z} + c$.
	\end{lemma}
	
	From Lemma~\ref{lemma equality}, if $\hat{z}$ is the optimal solution of Problem~\eqref{Section3-1}, we have
	\begin{align*}
		H_{\alpha} g(\hat{z}) = 0. \label{Halpha_condition}
	\end{align*}
	
	Next, we consider the dual problem corresponding to Problem~\eqref{realPro}.  
	The Lagrangian function for Problem~\eqref{realPro} is defined as
	\begin{align}
		\mathcal{L}(z, \mu) = f(z) + \mu^{T} s(z),
	\end{align}
	where $\mu$ denotes the Lagrange multiplier.
	
	Tsaknakis et al.~\cite{tsaknakisMinimaxProblemsCoupled2023} developed the duality theorem for minimax problems.  
	For Problem~\eqref{a1.1}, if the objective function $f(x, y)$ is strongly concave in $y$ for every $x$, then
	\begin{align}
		\min_{x} \max_{y} \min_{\mu \leq 0} \mathcal{L}(x, y, \mu)
		= \min_{\mu \leq 0} \min_{x} \max_{y} \mathcal{L}(x, y, \mu).
	\end{align}
	This implies that if the dual problem is well-defined, the associated multipliers satisfy $\mu \leq 0$.
	
	From the definitions of $N_{\alpha}^{\ast}$ and $H_{\alpha}$ and Corollary~\ref{corollary}, we can derive
	\begin{align*}
		\mu = - N_{\alpha}^{\ast} g(\bar{z})
	\end{align*}
	for Problem~\eqref{Section3-1} in Lemma~\ref{lemma equality}.  
	Accordingly, we define
	\begin{align*}
		u(\hat{z}) := - N_{\alpha}^{\ast} g(\hat{z})
	\end{align*}
	to denote the dual variable of Problem~\eqref{realPro}.  
	
	Hence, the conditions
	\begin{align}
		H_{\alpha} g(\hat{z}) = 0, \qquad u(\hat{z}) \leq 0,
	\end{align}
	are necessary for $\hat{z}$ to be the optimal solution of Problem~\eqref{Section3-1}.
	
	As the operators $N_{\alpha}^{\ast}$ and $H_{\alpha}$ play a fundamental role in the dual algorithm described in the next section, we now present several of their key properties that will be useful later.
	\begin{proportion}\label{Proportion of H}
		\begin{enumerate}
			\item $N_{\alpha}^{\ast}N_{\alpha} = I$,
			\item $H_{\alpha}$ is a symmetric matrix,
			\item $H_{\alpha}N_{\alpha} = 0$,
			\item $H_{\alpha}GH_{\alpha} = H_{\alpha}$,
			\item $N_{\alpha}^{\ast}GH_{\alpha} = 0$,
			\item $H_{\alpha_{+}}GH_{\alpha} = H_{\alpha_{+}}$.
		\end{enumerate}
	\end{proportion}
	
	\section{Dual algorithm}\label{section4}
	The algorithm presented in this section is based on the Basic Method~\ref{Basic Method} and employs the operators $H$ and $N^{\ast}$ defined in the previous section.  
	Details regarding its numerically stable implementation will be discussed in the next section.
	\begin{algorithm}{Dual algorithm:}\label{dual algorithm}
		\begin{enumerate}
			\item[Step 0] Find the unconstrained minimax point and preprocess problem:
			\begin{enumerate}
				\item Set $z\leftarrow -G^{-1}c, f\leftarrow\frac{1}{2}c^{T}z,H\leftarrow G^{-1},\alpha\leftarrow\emptyset,q\leftarrow0$.
				\item Compute $n_{i}^{T}Hn_{i}$ for all $i \in [m]$. Set $K \leftarrow \left\lbrace i\in[m] \mid n_{i}^{T}Hn_{i} < 0\right\rbrace$.
			\end{enumerate}
			\item[Step 1] Choose a violated constraint, if any:
			
			Compute $s_{j}(z)$, for all $j\in K\backslash\alpha$.\\
			If $V = \left\lbrace j\in K\backslash\alpha \mid s_{j}(z)>0\right\rbrace=\emptyset$, STOP. The current solution $z$ is both feasible and optimal;\\
			otherwise, choose $p\in V$ and set $n^{+} \leftarrow n_{p}$ and $u^{+} \leftarrow \begin{bmatrix}
				u\\0
			\end{bmatrix}$.
			
			If $q = 0$, set $u^{+}\leftarrow0$. ($\alpha_{+} = \alpha\cup\left\lbrace p\right\rbrace $)
			
			\item[Step 2] Check for feasibility and determine a new S-pair:
			
			\begin{enumerate}
				\item Determine step direction
				
				Compute $d = Hn^{+}$ (the step direction in the primal space) and if $q>0, r = -N^{\ast}n^{+}$ (the negative of the step direction in the dual space).
				
				\item Compute step length
				
				\begin{enumerate}
					\item Dual step length $t_{1}$; (maximum step in the dual space without violating dual feasibility).
					\begin{align*}
						t_{1} \leftarrow \min\left\lbrace\min_{\substack{j \in\alpha\\
								r_{j}<0}}\left\lbrace \frac{u_{j}^{+}}{r_{j}}\right\rbrace, \infty\right\rbrace,
					\end{align*}
					\item Primal step length $t_{2}$; (minimum step in primal space such that the $p$th constraint becomes feasible).
					\begin{align*}
						t_{2}\leftarrow\left\lbrace\begin{aligned}
							&\infty, &\text{if }d^{T}n^{+} \geq 0,\\
							&-\frac{s_{p}(z)}{d^{T}n^{+}}, &\text{otherwise}.
						\end{aligned}\right.
					\end{align*}
					\item Step length $t$:
					\begin{align*}
						t \leftarrow\min\left\lbrace t_{1}, t_{2}\right\rbrace.
					\end{align*}
				\end{enumerate}
				\item Determine new S-pair and take step:
				\begin{enumerate}
					\item No step in primal or dual space:
					
					If $t = \infty$, STOP, subproblem $P(\alpha_{+})$ and Problem~\eqref{realPro} are infeasible.
					
					\item Step in dual space:
					
					If $t_{2} = \infty$, then set $u^{+} \leftarrow u^{+} - t \begin{bmatrix}
						r\\
						1
					\end{bmatrix}$ and drop constraint $k$; namely, set $\alpha \leftarrow \alpha\backslash\left\lbrace k\right\rbrace$, $q \leftarrow q - 1$, update $H$ and $N^{\ast}$ and go to Step2(a)
					
					\item Step in primal and dual space:
				
					Set
					\begin{align*}
						\left\lbrace\begin{aligned}
							&z \leftarrow z + td,\\
							&f \leftarrow f + d^{T} n^{+} \left(\frac{1}{2}t - u_{p}^{+}\right),\\
							&u^{+} \leftarrow u^{+} - t\begin{bmatrix}
								r\\1
							\end{bmatrix} .
						\end{aligned}\right.
					\end{align*}
				\end{enumerate}
				If $t = t_{2}$ (full step), set $u\leftarrow u^{+},\alpha\leftarrow\alpha\cup\left\lbrace p\right\rbrace,q\leftarrow q + 1$, update $H,N^{\ast}$, and go to Step1.
				
				If $t = t_{1}$ (partial step), drop constraint $k$, set $\alpha\leftarrow\alpha\backslash\left\lbrace k\right\rbrace, q\leftarrow q-1$, update $H,N^{\ast}$, and go to Step2(a).
				\end{enumerate}
		\end{enumerate}
	\end{algorithm}
	
	The core of the dual algorithm lies in determining a new S-pair $(\bar{z}, \bar{\alpha} \cup \{ p \})$ in Step~2, given an existing S-pair $(z, \alpha)$ and a violated constraint $p$.  
	Unlike in the standard quadratic minimization problem, for minimax problems one can easily construct examples whose feasible regions are nonempty but for which no local minimax point exists, even in the convex–concave case.  
	As we shall see shortly, the following assumption is useful.
	
	\begin{assumption}\label{assumption}
		If Problem~\eqref{realPro} admits a local minimax point, then every subproblem $P(J)$ of Problem~\eqref{realPro}, where $J \subseteq K$, also admits a local minimax point.
	\end{assumption}
	
	It is evident that Assumption~\ref{assumption} is difficult to verify directly for Problem~\eqref{realPro}.  
	Therefore, we propose the following assumption, which is easier to check in practice.
	
	\begin{assumption}\label{assumption2}
		For Problem~\eqref{realPro}, the matrix $D G^{-1} D^{T}$ is negative semi-definite.
	\end{assumption}
	
	If Problem~\eqref{realPro} satisfies Assumption~\ref{assumption2}, then it necessarily satisfies Assumption~\ref{assumption}.  
	To illustrate the relationship between an S-pair $(z, \alpha)$ and a violated constraint $p$, we introduce the following definition.
	
	\begin{definition}\label{V-triple}
		A triple $(z, \alpha, p)$ consisting of a point $z$ and a set of indices $\alpha^{+} = \alpha \cup \{ p \}$, where $p \in K \setminus \alpha$, is said to be a \emph{V-triple} (violated triple) if the columns of $B_{\alpha}^{+}$ are linearly independent and the following conditions hold:
		\begin{enumerate}
			\item $s_{i}(z) = 0$ for all $i \in \alpha$ and $s_{p}(z) > 0$;
			\item $H_{\alpha_{+}} g(z) = 0$;
			\item $u_{\alpha_{+}} = - N_{\alpha_{+}}^{\ast} g(z) \le 0$;
			\item $\Gamma_{\alpha} \succ 0$.
		\end{enumerate}
	\end{definition}
	
	For an S-pair $(z, \alpha)$, we have $H_{\alpha} g(z) = 0$.  
	Using Proposition~\ref{Proportion of H}, it follows that
	\begin{align*}
	H_{\alpha_{+}} g(z) = H_{\alpha_{+}} G H_{\alpha} g(z) = 0.
	\end{align*}
	The main difference between a V-triple and an S-pair lies in the value of $s_{p}(z)$.  
	The key question is how to move from a point $z$ corresponding to a V-triple $(z, \alpha, p)$ to a new point while maintaining all the other conditions except for $s_{p}(z)$.  
	We now establish a lemma to clarify this process.
	
	\begin{lemma}\label{lemma of z}
		Let $(z, \alpha, p)$ be a V-triple and consider points of the form
		\begin{align}\label{definition of z-}
			\bar{z} = z + td,
		\end{align}
		where
		\begin{align}\label{definition of d}
			d = H_{\alpha}n^{+}.
		\end{align}
		Then
		\begin{align}
			& H_{\alpha_{+}} g(\bar{z}) = 0,\\
			& s_{i}(\bar{z}) = 0, \quad \text{for all } i \in \alpha,\\
			& s_{p}(\bar{z}) = s_{p}(z) + td^{T}n^{+},\\
			& u_{\alpha_{+}}(\bar{z}) = -N_{\alpha_{+}}^{\ast} g(\bar{z}) = u_{\alpha_{+}}(z) - t
			\begin{bmatrix}
				r\\
				1
			\end{bmatrix}, \label{definition of u}
		\end{align}
		where
		\begin{align}\label{definition of r}
			r = -N_{\alpha}^{\ast}n^{+},
		\end{align}
		and
		\begin{align}
			f(\bar{z}) = f(z) + td^{T}n^{+}\left(\frac{t}{2} - u_{p}(z)\right).
		\end{align}
	\end{lemma}
	
	\begin{proof}
		According to Proposition~\ref{Proportion of H} and Definition~\ref{V-triple}, we have
		\begin{align*}
			& g(\bar{z}) = G\bar{z} + c = Gz + c + tGd = g(z) + tGd,\\
			& Gd = GH_{\alpha}n^{+} = (I - N_{\alpha}N_{\alpha}^{\ast})n^{+}
			= N_{\alpha_{+}}
			\begin{bmatrix}
				-N_{\alpha}^{\ast}n^{+}\\
				1
			\end{bmatrix}.
		\end{align*}
		Then
		\begin{align*}
			& H_{\alpha_{+}}g(\bar{z}) = H_{\alpha_{+}}g(z) + tH_{\alpha_{+}}Gd = 0,\\
			& s_{i}(\bar{z}) = s_{i}(z) + tn_{i}^{T}H_{\alpha}n^{+} = 0,
			\quad \text{for all } i \in \alpha,\\
			& u_{\alpha_{+}}(\bar{z}) = -N_{\alpha_{+}}^{\ast}g(\bar{z})
			= -N_{\alpha_{+}}^{\ast}g(z) - tN_{\alpha_{+}}^{\ast}Gd
			= u_{\alpha_{+}}(z) - t
			\begin{bmatrix}
				r\\
				1
			\end{bmatrix}.
		\end{align*}
		From $H_{\alpha_{+}}g(z) = 0$, we obtain
		\begin{align*}
			g(z) = -N_{\alpha_{+}}u_{\alpha_{+}}(z).
		\end{align*}
		Therefore,
		\begin{align*}
			f(\bar{z})
			&= \frac{1}{2}\bar{z}^{T}G\bar{z} + c^{T}\bar{z} \\
			&= f(z) + \frac{1}{2}t^{2}d^{T}Gd + td^{T}g(z) \\
			&= f(z) + \frac{1}{2}t^{2}d^{T}n^{+} - td^{T}N_{\alpha_{+}}u_{\alpha_{+}}(z) \\
			&= f(z) + td^{T}n^{+}\left(\frac{t}{2} - u_{p}(z)\right).
		\end{align*}
	\end{proof}
	
	For a V-triple $(z, \alpha, p)$, if we want to obtain a new S-pair
	$(\bar{z}, \bar{\alpha}\cup\{p\})$, where $\bar{\alpha}$ is a subset of $\alpha$,
	it is necessary to verify whether $\Gamma_{\bar{\alpha}\cup\{p\}}$
	is a positive definite matrix. We now prove a lemma that establishes several
	properties of the matrix $\Gamma_{\alpha}$, which will be useful for
	constructing and analyzing the subsequent algorithm.

	\begin{lemma}\label{lemma positive}
		Let $\Gamma_{\alpha}$ be a positive definite matrix, $G_{22}$ be a negative definite matrix, and $n$ be a vector. Suppose that $H_{\alpha}$ is defined as in~\eqref{definition of H}. Then, the following properties hold:
		\begin{enumerate}
			\item If $\bar{\alpha} \subseteq \alpha$, then
			\begin{align}
				n^{T}H_{\alpha}n \geq n^{T}H_{\bar{\alpha}}n
				\quad\text{and}\quad
				\Gamma_{\bar{\alpha}} \succeq \Gamma_{\alpha} \succ 0.
			\end{align}
			\item If $n^{T}H_{\alpha}n < 0$, then $\Gamma_{\alpha_{+}} \succ 0$.
		\end{enumerate}
	\end{lemma}
	
	\begin{proof}
		For convenience, define
		\begin{align*}
			M_{\alpha} := 
			\begin{bmatrix}
				G_{12}^{T}\\
				A_{\alpha}
			\end{bmatrix},
			\qquad
			K_{\alpha} :=
			\begin{bmatrix}
				G_{22} & B_{\alpha}^{T}\\
				B_{\alpha} & 0
			\end{bmatrix}.
		\end{align*}
		By taking the Schur complement of $K_{\alpha}$, we obtain
		\begin{align*}
			\Gamma_{\alpha}
			= \Gamma_{\alpha_{-}}
			+ (B_{k}H_{\alpha_{-}}^{B}B_{k}^{T})^{-1}
			\left(A_{k}
			- \begin{bmatrix}
				B_{k} & 0
			\end{bmatrix}K_{\alpha_{-}}^{-1}M_{\alpha_{-}}\right)^{T}
			\left(A_{k}
			- \begin{bmatrix}
				B_{k} & 0
			\end{bmatrix}K_{\alpha_{-}}^{-1}M_{\alpha_{-}}\right),
		\end{align*}
		where $k = \alpha \setminus \alpha_{-}$ and
		\begin{align*}
		H^{B} := G_{22}^{-1}\left(I - B^{T}(BG_{22}^{-1}B^{T})^{-1}BG_{22}^{-1}\right).
		\end{align*}
		Since $G_{22}$ is negative definite, $H^{B}$ is negative semi-definite, which implies
		\begin{align*}
			\Gamma_{\alpha} \preccurlyeq \Gamma_{\alpha_{-}}.
		\end{align*}
		Hence, for $\bar{\alpha} \subseteq \alpha$, we have $\Gamma_{\bar{\alpha}} \succeq \Gamma_{\alpha}$.
		
		Next, we show that
		\begin{align*}
			\Gamma_{\alpha}\succ 0
			\quad \Longleftrightarrow \quad
			N_{\alpha}^{T}G^{-1}N_{\alpha} \prec 0,
		\end{align*}
		under the assumptions that $\Gamma_{\emptyset} \succ 0$ and $G_{22} \prec 0$.
		By the Woodbury matrix identity~\cite{highamAccuracyStabilityNumerical2002}, we have
		\begin{align*}
			\Gamma_{\alpha}^{-1}
			&= (\Gamma_{\emptyset} + W_{\alpha}^{T}U_{\alpha}^{-1}W_{\alpha})^{-1}\\
			&= \Gamma_{\emptyset}^{-1}
			- \Gamma_{\emptyset}^{-1}W_{\alpha}
			(U_{\alpha} + W_{\alpha}\Gamma_{\emptyset}^{-1}W_{\alpha}^{T})^{-1}
			W_{\alpha}^{T}\Gamma_{\emptyset}^{-1}\\
			&= \Gamma_{\emptyset}^{-1}
			- \Gamma_{\emptyset}^{-1}W_{\alpha}
			(N_{\alpha}^{T}G^{-1}N_{\alpha})^{-1}
			W_{\alpha}^{T}\Gamma_{\emptyset}^{-1},
		\end{align*}
		and
		\begin{align*}
			(N_{\alpha}^{T}G^{-1}N_{\alpha})^{-1}
			&= (U_{\alpha} + W_{\alpha}\Gamma_{\emptyset}^{-1}W_{\alpha}^{T})^{-1}\\
			&= U_{\alpha}^{-1}
			- U_{\alpha}^{-1}W_{\alpha}^{T}\Gamma_{\alpha}^{-1}W_{\alpha}U_{\alpha}^{-1},
		\end{align*}
		where $W_{\alpha} := A_{\alpha} - B_{\alpha}G_{22}^{-1}G_{12}^{T}$ and
		$U_{\alpha} := B_{\alpha}G_{22}^{-1}B_{\alpha}^{T}$.
		Since $G_{22} \prec 0$, it follows that $\Gamma_{\alpha} \succ 0$ if and only if
		$N_{\alpha}^{T}G^{-1}N_{\alpha} \prec 0$.
		
		Now consider the value $n^{T}H_{\bar{\alpha}}n - n^{T}H_{\alpha}n$,
		\begin{align*}
			n^{T}H_{\bar{\alpha}}n - n^{T}H_{\alpha}n
			&= -n^{T}G^{-1}N_{\bar{\alpha}}N_{\bar{\alpha}}^{\ast}n
			+ n^{T}G^{-1}N_{\alpha}N_{\alpha}^{\ast}n\\
			&= n^{T}G^{-1}N_{\alpha}
			\left(
			\begin{bmatrix}
				N_{\bar{\alpha}}^{T}G^{-1}N_{\bar{\alpha}} & N_{\bar{\alpha}}^{T}G^{-1}N_{\beta}\\
				N_{\beta}^{T}G^{-1}N_{\bar{\alpha}} & N_{\beta}^{T}G^{-1}N_{\beta}
			\end{bmatrix}^{-1}
			- 
			\begin{bmatrix}
				(N_{\bar{\alpha}}^{T}G^{-1}N_{\bar{\alpha}})^{-1} & 0\\
				0 & 0
			\end{bmatrix}
			\right)
			N_{\alpha}^{T}G^{-1}n,
		\end{align*}
		where $\beta = \alpha \setminus \bar{\alpha}$.
		Since both $N_{\alpha}^{T}G^{-1}N_{\alpha}$ and
		$N_{\bar{\alpha}}^{T}G^{-1}N_{\bar{\alpha}}$ are negative definite, using Cholesky decomposition, we define $N_{\bar{\alpha}}^{T}G^{-1}N_{\bar{\alpha}} = -L_{\bar{\alpha}}L_{\bar{\alpha}}^{T}$. Then the Cholesky decomposition of $N_{\alpha}^{T}G^{-1}N_{\alpha}$ can be written as
		\begin{align*}
			-N_{\alpha}^{T}G^{-1}N_{\alpha} = -\begin{bmatrix}
				L_{\bar{\alpha}} & 0\\
				N_{\beta}^{T}G^{-1}N_{\bar{\alpha}}L_{\bar{\alpha}}^{-T} & L_{\beta}
			\end{bmatrix}\begin{bmatrix}
				L_{\bar{\alpha}}^{T} & L_{\bar{\alpha}}^{-1}N_{\bar{\alpha}}^{T}G^{-1}N_{\beta}\\
				0 & L_{\beta}^{T}
			\end{bmatrix},
		\end{align*}
		where $L_{\beta}L_{\beta}^{T} = -N_{\beta}^{T}H_{\bar{\alpha}}N_{\beta}$. By taking the Schur complement of $\begin{bmatrix}
			L_{\bar{\alpha}} & 0\\
			N_{\beta}^{T}G^{-1}N_{\bar{\alpha}}L_{\bar{\alpha}}^{-T} & L_{\beta}
		\end{bmatrix}$, we obtain
		\begin{align*}
			\begin{bmatrix}
				L_{\bar{\alpha}} & 0\\
				N_{\beta}^{T}G^{-1}N_{\bar{\alpha}}L_{\bar{\alpha}}^{-T} & L_{\beta}
			\end{bmatrix}^{-1} = \begin{bmatrix}
				L_{\bar{\alpha}}^{-1} & 0\\
				-L_{\beta}^{-1}N_{\beta}^{T}G^{-1}N_{\bar{\alpha}}L_{\bar{\alpha}}^{-T}L_{\bar{\alpha}}^{-1} & L_{\beta}^{-1}
			\end{bmatrix}.
		\end{align*}
		Hence,
		\begin{align*}
			\begin{bmatrix}
				N_{\bar{\alpha}}^{T}G^{-1}N_{\bar{\alpha}} & N_{\bar{\alpha}}^{T}G^{-1}N_{\beta}\\
				N_{\beta}^{T}G^{-1}N_{\bar{\alpha}} & N_{\beta}^{T}G^{-1}N_{\beta}
			\end{bmatrix}^{-1} - \begin{bmatrix}
				(N_{\bar{\alpha}}^{-1}G^{-1}N_{\bar{\alpha}})^{-1} & 0\\
				0 & 0
			\end{bmatrix} & = -\begin{bmatrix}
				L^{T}L & -L^{T}L_{\beta}^{-1}\\
				-L_{\beta}^{-T}L & L_{\beta}^{-T}L_{\beta}^{-1}
			\end{bmatrix}\\
			& = -\begin{bmatrix}
				L^{T}\\
				-L_{\beta}^{-T}
			\end{bmatrix} \begin{bmatrix}
				L & -L_{\beta}^{-1}
			\end{bmatrix},
		\end{align*}
		where $L = L_{\beta}^{-1}N_{\beta}^{T}G^{-1}N_{\bar{\alpha}}L_{\bar{\alpha}}^{-T}L_{\bar{\alpha}}^{-1}$. This shows that $ \left(\begin{bmatrix}
			N_{\bar{\alpha}}^{T}G^{-1}N_{\bar{\alpha}} & N_{\bar{\alpha}}^{T}G^{-1}N_{\beta}\\
			N_{\beta}^{T}G^{-1}N_{\bar{\alpha}} & N_{\beta}^{T}G^{-1}N_{\beta}
		\end{bmatrix}^{-1} - \begin{bmatrix}
			(N_{\bar{\alpha}}^{-1}G^{-1}N_{\bar{\alpha}})^{-1} & 0\\
			0 & 0
		\end{bmatrix}\right)$ is negative semi-definite, and thus $n^{T}H_{\bar{\alpha}}n \leq n^{T}H_{\alpha}n$.
		
		Finally, we define $N_{\alpha}^{T}G^{-1}N_{\alpha} = -L_{\alpha}L_{\alpha}^{T}$.
		If $n^{T}H_{\alpha}n < 0$, then
		\begin{align*}
			N_{\alpha_{+}}^{T}G^{-1}N_{\alpha_{+}}
			&=
			\begin{bmatrix}
				N_{\alpha}^{T}G^{-1}N_{\alpha} & N_{\alpha}^{T}G^{-1}n\\
				n^{T}G^{-1}N_{\alpha} & n^{T}G^{-1}n
			\end{bmatrix}\\
			&=
			-\begin{bmatrix}
				L_{\alpha} & 0\\
				n^{T}G^{-1}N_{\alpha}L_{\alpha}^{-T} & \sqrt{-n^{T}H_{\alpha}n}
			\end{bmatrix}
			\begin{bmatrix}
				L_{\alpha}^{T} & L_{\alpha}^{-1}N_{\alpha}^{T}G^{-1}n\\
				0 & \sqrt{-n^{T}H_{\alpha}n}
			\end{bmatrix},
		\end{align*}
		which implies that $N_{\alpha_{+}}^{T}G^{-1}N_{\alpha_{+}} \prec 0$ and, consequently,
		\begin{align*}
		\Gamma_{\alpha_{+}} \succ 0.
		\end{align*}
	\end{proof}
	
	Assume that $(z^{\ast}, \alpha^{\ast})$ is an S-pair for Problem~\eqref{realPro}.
	From Lemma~\ref{lemma positive} and Corollary~\ref{corollary}, we have
	$\Gamma_{p} \succ 0$ for all $p \in \alpha^{\ast}$, which is equivalent to
	$n_{p}^{T}G^{-1}n_{p} < 0$.
	This implies that $p$th constraint can be added to the active set only if
	$n_{p}^{T}G^{-1}n_{p} < 0$.
	In Step~0 of the dual algorithm~\ref{dual algorithm}, we compute
	$n_{i}^{T}G^{-1}n_{i}$ for all $i \in [m]$, and define
	\begin{align}
	K := \left\lbrace i \in [m] \mid n_{i}^{T}G^{-1}n_{i} < 0 \right\rbrace,
	\end{align}
	which represents the indices of constraints that may enter the active set.
	
	Lemma~\ref{lemma of z} describes how to determine the search direction $d$,
	while Lemma~\ref{lemma positive} ensures the positive definiteness of
	$\Gamma_{\alpha}$. In the next step, we present the procedure for selecting the
	step size $t$ at each iteration.
	
	\begin{theorem}\label{V-tri z}
		Let $(z, \alpha, p)$ be a V-triple. Define $\bar{z}$ by \eqref{definition of z-} and \eqref{definition of d}, and let
		\begin{align}\label{definition of t}
			t = \min\{t_{1},\, t_{2}\},
		\end{align}
		where
		\begin{align}\label{definition of t1}
			t_{1} = \min\left\lbrace\min_{\substack{j \in \alpha\\ r_{j} < 0}} \left\lbrace\frac{u_{j}^{+}}{r_{j}}\right\rbrace, \infty \right\rbrace,
		\end{align}
		and
		\begin{align}\label{definition of t2}
			t_{2} =
			\left\lbrace\begin{aligned}
				&\infty, & \text{if } d^{T} n^{+} \geq 0,\\
				&-\dfrac{s_{p}(z)}{d^{T}n^{+}}, & \text{otherwise.}
			\end{aligned}\right.
		\end{align}
		Then the following statements hold:
		\begin{enumerate}
			\item If $t = t_{1} = u_{k}^{+}(z) / r_{k}$, then $(\bar{z},\, \alpha \setminus \{k\},\, p)$ is a V-triple.
			\item If $t = t_{2} = -s_{p}(z)/(d^{T}n^{+})$, then $(\bar{z},\, \alpha \cup \{p\})$ is an S-pair.
			\item If $t = \infty$, then $P(\alpha \cup \{p\})$ is infeasible.
		\end{enumerate}
	\end{theorem}
	
	\begin{proof}
		By the definition of $t$ in \eqref{definition of t}, \eqref{definition of t1}, \eqref{definition of t2} and Lemma~\ref{lemma of z}, we have
		\begin{align*}
		H^{+} g(\bar{z}) = 0, \qquad
		s_{i}(\bar{z}) = 0~(\forall i \in \alpha), \qquad
		u^{+}(\bar{z}) \le 0.
		\end{align*}

		\noindent\textbf{Case 1:} $t = t_{1} < t_{2}$.  
		Then $u_{k}^{+}(\bar{z}) = 0$ and $s_{p}(\bar{z}) > 0$. Since $H^{+}g(\bar{z}) = 0$ and $u_{k}^{+}(\bar{z}) = 0$, we have
		\begin{align*}
		g(\bar{z}) = N^{+}u^{+}(\bar{z}) = 
		\sum_{i \in \alpha \cup \{p\} \setminus \{k\}} u_{i}^{+} n_{i}.
		\end{align*}
		The set of normals $\{ b_{i} \mid i \in \alpha \cup \{p\} \setminus \{k\} \}$ is linearly independent.  
		By Lemma~\ref{lemma positive}, the corresponding $\Gamma_{\alpha\setminus\{k\}}$ is positive definite. Hence $(\bar{z},\, \alpha \setminus \{k\},\, p)$ is a V-triple.
		
		\noindent\textbf{Case 2:} $t = t_{2} = -s_{p}(z)/(d^{T}n^{+})$.  
		Then $s_{p}(\bar{z}) = 0$.  
		By Lemma~\ref{lemma positive}, we have $\Gamma_{\alpha_{+}} \succ 0$ and thus $(\bar{z},\, \alpha \cup \{p\})$ is an S-pair.

		\noindent\textbf{Case 3:} $t = \infty$.  
		In this case $r \ge 0$ and $n^{+T}Hn^{+} \ge 0$.  
		Since $(z, \alpha, p)$ is a V-triple, the following system holds:
		\begin{align*}
		\begin{cases}
			Gz + c + N u = 0,\\
			N^{T}z + h = 0,\\
			n^{+T}z + h^{+} > 0.
		\end{cases}
		\end{align*}
		Assume, for contradiction, that Problem $P(\alpha \cup \{p\})$ admits a local minimax point $(z^{\ast}, u^{\ast})$.  
		Then
		\begin{align*}
		\begin{cases}
			Gz^{\ast} + c + N^{+}u^{\ast} = 0,\\
			N^{T}z^{\ast} + h \le 0,\\
			n^{+T}z^{\ast} + h^{+} \le 0.
		\end{cases}
		\end{align*}
		Let $d = z^{\ast} - z$.  
		Subtracting the two systems gives
		\begin{align}
			Gd + N^{+}u^{\ast} - N u &= 0, \label{eq:sys1}\\
			N^{T}d &\le 0, \label{eq:sys2}\\
			n^{+T}d &< 0. \label{eq:sys3}
		\end{align}
		From \eqref{eq:sys1} and \eqref{eq:sys2}, we have
		\begin{align*}
			d &= G^{-1}(N u - N^{+}u^{\ast}),\\
			(N^{T}G^{-1}N)^{-1}N^{T}d
			&= u - u_{\alpha}^{\ast} - (N^{T}G^{-1}N)^{-1}N^{T}G^{-1}n^{+}u_{p}^{\ast}.
		\end{align*}
		Consequently,
		\begin{align*}
			n^{+T}d
			&= n^{+T}G^{-1}(N u - N^{+}u^{\ast})\\
			&= -n^{+T}G^{-1}n^{+}u_{p}^{\ast}
			+ n^{+T}G^{-1}N(N^{T}G^{-1}N)^{-1}N^{T}G^{-1}n^{+}u_{p}^{\ast}
			+ n^{+T}G^{-1}N(N^{T}G^{-1}N)^{-1}N^{T}d\\
			&= -n^{+T}Hn^{+}u_{p}^{\ast} - r^{T}N^{T}d\\
			&\geq 0,
		\end{align*}
		which contradicts \eqref{eq:sys3}.  
		Therefore, Problem $P(\alpha \cup \{p\})$ is infeasible.
	\end{proof}
	
	Noting that $G$ is an indefinite matrix, there exists a vector $n^{+}$ such that $n^{+T}Hn \geq 0$. According to Theorem~\ref{V-tri z}, the value of $n^{+T}Hn$ plays a crucial role, as it determines the value of $t_{2}$ and $t$. We now present two sufficient conditions on $n^{+}$ that ensure $n^{+T}Hn \geq 0$. Moreover, if $n^{+}$ does not satisfy two conditions below, $n^{+T}Hn^{+}$ may still be non-negative.
	
	\begin{lemma}\label{lemma linear}
		Let $(z, \alpha)$ be an S-pair.  
		If $n_{p}$ is a linear combination of the columns of $N_{\alpha}$, then $n_{p}^{T}H_{\alpha}n_{p} = 0$.  
		If $B_{p}$ is a linear combination of the rows of $B_{\alpha}$, then $n_{p}^{T}H_{\alpha}n_{p} \ge 0$.
	\end{lemma}
	
	\begin{proof}
		If $n_{p}$ is a linear combination of the columns of $N_{\alpha}$, then there exists a vector $\beta$ such that 
		\begin{align*}
		n_{p} = N_{\alpha}\beta = \sum_{i \in \alpha} n_{i}\beta_{i}.
		\end{align*}
		By Proposition~\ref{Proportion of H}, we have $H_{\alpha}n_{p} = 0$, hence $n_{p}^{T}H_{\alpha}n_{p} = 0$.
		
		For the second claim, from the definition of $H_{\alpha}$ we obtain
		\begin{align*}
			n_{p}^{T}H_{\alpha}n_{p}
			&= B_{p}G_{22}^{-1}B_{p}^{T} + W_{p}F^{-1}W_{p}^{T} 
			- (W_{p}F^{-1}W_{\alpha}^{T} + B_{p}G_{22}^{-1}B_{\alpha}^{T})(N_{\alpha}^{T}G^{-1}N_{\alpha})^{-1}
			(W_{\alpha}F^{-1}W_{p}^{T} + B_{\alpha}G_{22}^{-1}B_{p}^{T}) \\
			&= B_{p}H_{\alpha}^{B}B_{p}^{T} 
			+ (W_{p} - B_{p}G_{22}^{-1}B_{\alpha}^{T}U_{\alpha}^{-1}W_{\alpha})
			\Gamma_{\alpha}^{-1}
			(W_{p}^{T} - W_{\alpha}^{T}U_{\alpha}^{-1}B_{\alpha}G_{22}^{-1}B_{p}^{T}),
		\end{align*}
		where 
		\begin{align*}
		F = G_{11} - G_{12}G_{22}^{-1}G_{12}^{T}, \quad
		W_{\alpha} = A_{\alpha} - B_{\alpha}G_{22}^{-1}G_{12}^{T}, \quad
		U_{\alpha} = B_{\alpha}G_{22}^{-1}B_{\alpha}^{T},
		\end{align*}
		and 
		\begin{align*}
		H_{\alpha}^{B} := G_{22}^{-1}\left(I - B_{\alpha}^{T}(B_{\alpha}G_{22}^{-1}B_{\alpha}^{T})^{-1}B_{\alpha}G_{22}^{-1}\right).
		\end{align*}
		Note that $H_{\alpha}^{B}$ shares the same structural property as $H_{\alpha}$ in Proportion~\ref{Proportion of H}.  
		If $B_{p}$ is a linear combination of the rows of $B_{\alpha}$, then $B_{p}H_{\alpha}^{B}B_{p}^{T} = 0$.  
		Since $\Gamma_{\alpha}$ is positive definite, it follows that $n_{p}^{T}H_{\alpha}n_{p} \ge 0$.
	\end{proof}
	
	We now establish a theorem that describes how to generate a new S-pair from a given one. 
	In particular, we analyze the behavior of the objective function and show that its value decreases after each update.
	
	\begin{theorem}\label{theorem <}
		Let $(z, \alpha)$ be an S-pair and let $p \in K \setminus \alpha$ be an index such that 
		\begin{align*}
		n_{p}^{T}H_{\alpha}n_{p} < 0 \quad \text{and} \quad s_{p}(z) > 0.
		\end{align*}
		Then a new S-pair $(\bar{z}, \bar{\alpha}\cup\{p\})$ can be obtained, where the step direction $d$ is defined by \eqref{definition of d}, $\bar{\alpha} \subseteq \alpha$, and 
		\begin{align*}
		f(\bar{z}) < f(z)
		\end{align*}
		after $|\alpha| - |\bar{\alpha}|$ partial steps and one full step.
	\end{theorem}
	
	\begin{proof}
		From Lemma~\ref{lemma positive}, Lemma~\ref{lemma linear}, and the condition $n_{p}^{T}H_{\alpha}n_{p} < 0$, it follows that $(z, \alpha, p)$ is a V-triple and $n_{p}^{T}H_{\bar{\alpha}}n_{p} < 0$ for some $\bar{\alpha} \subseteq \alpha$. 
		By Theorem~\ref{V-tri z} and Lemma~\ref{lemma of z}, each iteration either produces a new V-triple or an S-pair. 
		In each step, the objective value satisfies
		\begin{align*}
		f(\bar{z}) = f(z) + t d^{T}n^{+}\left(\tfrac{t}{2} - u_{p}^{+}(z)\right) \le 0.
		\end{align*}
		Since the active set $\alpha$ contains finitely many constraints and one constraint is dropped in each partial step, we eventually obtain a new S-pair $(\bar{z}, \bar{\alpha}\cup\{p\})$ in finitely many steps with $f(\bar{z}) < f(z)$.
	\end{proof}
	
	From Lemma~\ref{lemma positive}, the case $n_{p}^{T}H_{\alpha}n_{p} < 0$ is generic. 
	We now turn to the remaining case $n_{p}^{T}H_{\alpha}n_{p} \ge 0$ and show how to obtain a new S-pair or detect infeasibility, while again ensuring that the objective value decreases.
	
	\begin{theorem}\label{theorem >=}
		Let $(z, \alpha)$ be an S-pair and let $p \in K \setminus \alpha$ be an index such that
		\begin{align*}
		n_{p}^{T}H_{\alpha}n_{p} \ge 0 \quad \text{and} \quad s_{p}(z) > 0.
		\end{align*}
		Then a new S-pair $(\bar{z}, \bar{\alpha}\cup\{p\})$ can be obtained, where the step direction $d$ is defined by \eqref{definition of d}, $\bar{\alpha} \subseteq \alpha$, and $f(\bar{z}) < f(z)$ after $|\alpha| - |\bar{\alpha}|$ partial steps and one full step, 
		or $P(\bar{\alpha}\cup\{p\})$ is infeasible within at most $|\alpha|$ steps.
	\end{theorem}
	
	\begin{proof}
		Since $(z, \alpha)$ is an S-pair, we have $H_{\alpha}g(z) = 0$ and $H_{\alpha_{+}}g(z) = 0$. 
		Likewise \eqref{definition of d}, \eqref{definition of t1} and \eqref{definition of z-}, in each iteration we define
		\begin{align*}
			d_{i} & = H_{i}n^{+},\\
			t_{i} & = \min\left\lbrace\min_{\substack{j \in\alpha\\
					r_{j}<0}}\left\lbrace \frac{u_{j}^{+}}{r_{j}}\right\rbrace, \infty\right\rbrace.
		\end{align*}
		and
		\begin{align*}
			z_{i} = z_{i - 1} + t_{i}d_{i}.
		\end{align*}
		where $H_{i}$, $z_{i}$, $t_{i}$ and $d_{i}$ denote the corresponding matrices or vectors at the $i$-th iteration, with the initialization $z_{0} = z$.
		
		If $t_{i} < \infty$ for all steps, then by Proposition~\ref{Proportion of H} we have $H_{\alpha_{+}}g(z_{1}) = 0$. 
		Using $u_{k}^{+}(z_{1}) = 0$ when $t_{1} = u_{k}^{+}(z_{0})/r_{k}$, we get $H_{\alpha_{+}\setminus\{k\}}g(z_{1}) = 0$. 
		From Lemma~\ref{lemma positive}, we have $\Gamma_{\alpha_{-}} \succ 0$. Since $\alpha$ contains finitely many indices and one constraint is removed per iteration, we reach a point $\hat{z}$ such that $n_{p}^{T}H_{\hat{\alpha}}n_{p} < 0$ within at most $|\alpha|$ steps, with $H_{\hat{\alpha}^{+}}g(\hat{z}) = 0$ and $\Gamma_{\hat{\alpha}} \succ 0$. 
		Thus $(\hat{z}, \hat{\alpha}, p)$ forms a V-triple, and by Theorems~\ref{V-tri z} and~\ref{theorem <}, we obtain a new S-pair $(\bar{z}, \bar{\alpha}\cup\{p\})$ in finitely many steps.
		
		If $t_{i} = \infty$ for some $i$, infeasibility of a subproblem can be shown as in the proof of Theorem~\ref{V-tri z}.
		
		Finally, we analyze the change in $f(z)$. 
		Suppose $(\hat{z}, \hat{\alpha}, p)$ is obtained after $q_{1}$ steps and $(\bar{z}, \bar{\alpha}\cup\{p\})$ after $q_{1}+q_{2}$ steps. 
		From Lemma~\ref{lemma positive}, we have
		\begin{align*}
		n_{p}^{T}H_{1}n_{p} \ge n_{p}^{T}H_{2}n_{p} \ge \cdots \ge n_{p}^{T}H_{q_{1}}n_{p} \ge 0 > n_{p}^{T}H_{q_{1}+1}n_{p} \ge \cdots \ge n_{p}^{T}H_{q_{1}+q_{2}}n_{p}.
		\end{align*}
		Consequently,
		\begin{align*}
		s_{p}(\bar{z}) = s_{p}(z) + n_{p}^{T}\sum_{i=1}^{q_{1}+q_{2}} t_{i}d_{i} = 0,
		\quad s_{p}(z) > 0,
		\end{align*}
		and
		\begin{align*}
			f(\bar{z}) - f(z) = & n_{p}^{T}\big(t_{1}d_{1}\frac{t_{1}}{2} + t_{2}d_{2}(\frac{t_{2}}{2} + t_{1}) + \cdots + t_{q_{1}}d_{q_{1}}(\frac{t_{q_{1}}}{2} + \sum_{i = 1}^{q_{1} - 1}t_{i}) \\
			& + t_{q_{1} + 1}d_{q_{1} + 1}(\frac{t_{q_{1} + 1}}{2} + T) + \cdots + t_{q_{1} + q_{2}}d_{q_{1} + q_{2}}(\frac{t_{q_{1} + q_{2}}}{2} + T + \sum_{i = q_{1} + 1}^{q_{1} + q_{2}} t_{i})\big),			
		\end{align*}
		where $T = \sum_{i = 1}^{q_{1}}t_{i}$. Since $n_{p}^{T}(t_{1}d_{1} + t_{2}d_{2} + \cdots + t_{q_{1} + q_{2}}d_{q_{1} + q_{2}}) < 0$ and $n_{p}^{T}H_{i}n_{p} < 0$ for each $i = q_{1} + 1, \cdots, q_{1} + q_{2}$, we have
		\begin{align*}
			f(\bar{z}) - f(z) & < n_{p}^{T}\left(t_{q_{1} + 1}d_{q_{1} + 1}(\frac{t_{q_{1} + 1}}{2}) + \cdots + t_{q_{1} + q_{2}}d_{q_{1} + q_{2}}(\frac{t_{q_{1} + q_{2}}}{2} + \sum_{i = q_{1} + 1}^{q_{1} + q_{2}} t_{i})\right)\\
			& \leq 0.
		\end{align*}
		which implies $f(\bar{z}) < f(z)$.
	\end{proof}

	Each time Step~1 of the dual algorithm~\ref{dual algorithm} is executed, the current point $z$ solves the subproblem $P(\alpha)$, i.e., $(z, \alpha)$ is an S-pair. 
	If $z$ satisfies all constraints of Problem~\eqref{realPro}, then $z$ is a local minimax point. 
	Otherwise, after at most $|\alpha|$ partial steps and one full step, a new S-pair $(\bar{z}, \bar{\alpha})$ is obtained and the process returns to Step~1, or infeasibility is detected according to Theorem~\ref{theorem >=} under Assumption~\ref{assumption}. 
	Since Problem~\eqref{realPro} admits only finitely many S-pairs, and by Theorems~\ref{theorem <}, \ref{theorem >=} each new S-pair yields a strict decrease in the objective value, the algorithm does not cycle.
	
	\begin{theorem}
		Under Assumption~\ref{assumption}, the dual algorithm either solves Problem~\eqref{realPro} or detects infeasibility in a finite number of steps.
	\end{theorem}
	
	\section{Numerically stable implementation}\label{section5}
	
	There exist various approaches to implement the dual algorithm described in Section~\ref{section4} in a numerically stable manner. 
	Our implementation is based on the Cholesky decomposition and Givens rotation.
	
	Since $N^{T}G^{-1}N$ is negative definite, we define its Cholesky factorization as
	\begin{align}
		-N^{T}G^{-1}N = R^{T}R,
	\end{align}
	and introduce
	\begin{align}
		M = R^{-T}N^{T}G^{-1}.
	\end{align}
	Then the matrices $H$ and $N^{\ast}$ can be written as
	\begin{align}
		N^{\ast} &= (N^{T}G^{-1}N)^{-1}N^{T}G^{-1} = -R^{-1}M,\\
		H &= G^{-1}(I - NN^{\ast}) = G^{-1} + M^{T}M.
	\end{align}
	Hence, instead of explicitly storing $H$ and $N^{\ast}$, it suffices to maintain $R$, $R^{-1}$, and $M$.
	
	In the dual algorithm of Section~\ref{section4}, the vectors $d$ and $r$ are required. They can be computed as
	\begin{align*}
		r &= -N^{\ast}n^{+} = R^{-1}Mn^{+},\\
		d &= Hn^{+} = (G^{-1} + M^{T}M)n^{+}.
	\end{align*}
	In practice, we perform the following computations:
	\begin{align*}
		d_{1} &= Mn^{+},\\
		d_{2} &= (G^{-1} + M^{T}M)n^{+},\\
		\delta &= n^{+T}(G^{-1} + M^{T}M)n^{+},\\
		r &= R^{-1}d_{1}.
	\end{align*}
	
	\paragraph{Adding a constraint.}
	Let $N_{+}$ denote the matrix obtained by adding one constraint column $n^{+}$ to $N$, and define the Cholesky factorization
	\begin{align*}
		-N_{+}^{T}G^{-1}N_{+} = R_{+}^{T}R_{+},
	\end{align*}
	that is,
	\begin{align*}
		R_{+}^{T}R_{+} = -\begin{bmatrix}
			N^{T}G^{-1}N & N^{T}G^{-1}n^{+}\\
			n^{+T}G^{-1}N & n^{+T}G^{-1}n^{+}
		\end{bmatrix}.
	\end{align*}
	Following the proof in Lemma~\ref{lemma positive}, one obtains
	\begin{align*}
		R_{+} = 
		\begin{bmatrix}
			R & -d_{1}\\
			0 & \sqrt{-\delta}
		\end{bmatrix},
		\qquad
		R_{+}^{-1} = 
		\begin{bmatrix}
			R^{-1} & \frac{R^{-1}d_{1}}{\sqrt{-\delta}}\\
			0 & \frac{1}{\sqrt{-\delta}}
		\end{bmatrix}.
	\end{align*}
	Consequently,
	\begin{align*}
		M_{+} 
		&= R_{+}^{-T}N_{+}^{T}G^{-1}
		= 
		\begin{bmatrix}
			R^{-T}N^{T}G^{-1}\\
			\frac{d_{1}^{T}R^{-T}N^{T}G^{-1} + n^{+T}G^{-1}}{\sqrt{-\delta}}
		\end{bmatrix}
		=
		\begin{bmatrix}
			M\\
			\frac{d_{2}^{T}}{\sqrt{-\delta}}
		\end{bmatrix}.
	\end{align*}
	Thus, after computing $d_{1}$, $d_{2}$, $\delta$, and $r$, the matrices $R$, $R^{-1}$, and $M$ can be efficiently updated to $R_{+}$, $R_{+}^{-1}$, and $M_{+}$, respectively.
	
	\paragraph{Dropping a constraint.}
	We first consider dropping the last constraint, that is, removing the last column of $N$.  
	Let $N = [\,N_{-}\; n\,]$ and define $-N_{-}^{T}G^{-1}N_{-} = R_{-}^{T}R_{-}$.  
	Then
	\begin{align*}
		R = 
		\begin{bmatrix}
			R_{-} & \ast\\
			0 & \ast
		\end{bmatrix}, \quad
		R^{-1} = 
		\begin{bmatrix}
			R_{-}^{-1} & \ast\\
			0 & \ast
		\end{bmatrix}, \quad
		M = 
		\begin{bmatrix}
			M_{-}\\
			\ast
		\end{bmatrix}.
	\end{align*}
	Therefore, deleting the last row and column of $R$ and $R^{-1}$ yields $R_{-}$ and $R_{-}^{-1}$, respectively, while deleting the last row of $M$ gives $M_{-}$.
	
	Next, consider dropping the $k$th constraint.  
	Let $N = \begin{bmatrix}
		N_{1} & N_{D} & N_{2}
	\end{bmatrix}$, where $N_{D}$ is the column to be deleted, and define $N_{-} = \begin{bmatrix}
		N_{1} & N_{2}
	\end{bmatrix}$.  
	Let $R = \begin{bmatrix}
		R_{1} & R_{D} & R_{2}
	\end{bmatrix}$, where $R_{D}$ corresponds to the $k$th column of $R$.  
	There exists a elementary matrix $P$ such that
	\begin{align*}
	\hat{N} = NP = \begin{bmatrix}
		N_{1} & N_{2} & N_{D}
	\end{bmatrix},
	\end{align*}
	which implies
	\begin{align*}
		-P^{T}N^{T}G^{-1}NP = P^{T}R^{T}RP,
	\end{align*}
	and
	\begin{align*}
	RP = \begin{bmatrix}
		R_{1} & R_{2} & R_{D}
	\end{bmatrix}.
	\end{align*}
	Since $\begin{bmatrix}
		R_{1} & R_{2}
	\end{bmatrix}$ is an upper-Hessenberg matrix, an upper triangular matrix
	\begin{align*}
	\hat{R} = \bar{Q}^{T}RP = \bar{Q}^{T}\begin{bmatrix}
		R_{1} & R_{2} & R_{D}
	\end{bmatrix}
	\end{align*}
	can be obtained by applying at most $|\alpha|$ Givens rotations, where 
	\(
	\bar{Q} = Q_{k-1,k}\cdots Q_{2,3}Q_{1,2}
	\)
	is the product of the Givens matrices.  
	It follows that
	\begin{align*}
	\hat{R}^{-1} = (\bar{Q}^{T}RP)^{-1} = P^{T}R^{-1}\bar{Q},
	\end{align*}
	and
	\begin{align*}
		\hat{M} 
		&= \hat{R}^{-T}\hat{N}^{T}G^{-1}
		= \bar{Q}^{T}R^{-T}PP^{T}N^{T}G^{-1}
		= \bar{Q}^{T}R^{-T}N^{T}G^{-1}
		= \bar{Q}^{T}M.
	\end{align*}
	Thus, using the elementary matrix $P$ and the Givens rotations $\bar{Q}$, the $k$th constraint can be moved to the last column of $N$, after which the removal process is equivalent to dropping the last column of $\hat{N}$.  
	Accordingly, deleting the last row and column of $\hat{R}$ and $\hat{R}^{-1}$ gives $R_{-}$ and $R_{-}^{-1}$, while deleting the last row of $\hat{M}$ gives $M_{-}$.
	
	In Step~0(b) of the dual algorithm~\ref{dual algorithm}, we need to evaluate $n_{i}^{T}H^{-1}n_{i}$ for all $i \in [m]$.  
	In practice, we first compute $G^{-1}n_{i}$ and then obtain $n_{i}^{T}H^{-1}n_{i}$ to identify the index set $K$.  
	The vectors $G^{-1}n_{i}$ for $i \in K$ are stored for later use in computing the step direction $d$. At each iteration, we compute $d_{1} = Mn^{+}$, then evaluate
	\begin{align*}
	d_{2} = M^{T}d_{1} + G^{-1}n_{p},
	\end{align*}
	which avoids $(n_{x} + n_{y})^{2}$ operations compared with computing $d_{2}$ directly.  
	Subsequently, we calculate $\delta = n_{p}^{T}d_{2}$ and $r = R^{-1}d_{1}$.  
	The step direction $d$ is given by $d_{2}$, the dual step length $t_{1}$ is obtained from $u$ and $r$, and the primal step length $t_{2}$ is determined from $s_{p}$ and $\delta$.  
	After $d$ and $t$ are computed, we determine the step space and perform the update, followed by updating $R$, $R^{-1}$, and $M$ to proceed to the next iteration.
	
	\section{Numerical results}\label{section6}
	This section presents numerical experiments to evaluate the proposed dual algorithm from both theoretical and computational perspectives. The experiments assess its numerical stability and efficiency across different problem settings. Three types of experiments are considered: the first assesses basic numerical properties on randomly generated problems, illustrating its stability and computational effectiveness; the second applies the algorithm to an adversarial attack on a mean–covariance portfolio model, demonstrating its practical effectiveness; and the third presents two illustrative examples under distinct assumptions, highlighting the detailed iterative behavior of Algorithm~\ref{dual algorithm}.
	
	\subsection{Randomly generated test problems}
	In this subsection, we demonstrate the performance of the proposed dual algorithm on randomly generated quadratic minimax problems with coupled inequality constraints. Two classes of test problems were constructed such that their optimal solutions were known in advance, following the procedures inspired by Goldfarb and Idnani~\cite{goldfarbNumericallyStableDual1983} and Rosen and Suzuki~\cite{rosenPracniquesConstructionNonlinear1965}. 
	
	Each problem instance was characterized by the number of variables in $x$ and $y$, denoted by $n_{x}$ and $n_{y}$, respectively, the total number of constraints $m$, and the number of constraints in the active set at the local minimax point, $n_{a}$. The off-diagonal elements of $G_{22}$ were generated as $r(-1,1)$, and the diagonal elements were set to $-S_{i}^{-} - r(0,1) - 1$, where $r(a,b)$ denotes a uniformly distributed random number between $a$ and $b$, and $S_{i}^{-}$ represents the sum of the absolute values of the off-diagonal elements in the $i$th row of $G_{22}$. Each element of $G_{12}$ was set to $r(-1,1)$. The off-diagonal entries of $\Gamma_{\emptyset} = G_{11} - G_{12}G_{22}^{-1}G_{12}^{T}$ were set to $r(-1,1)$, while the diagonal elements were set to $S_{i}^{+} + r(0,1) + 1$, where $S_{i}^{+}$ denotes the sum of the absolute values of the off-diagonal elements in the $i$th row of $\Gamma_{\emptyset}$. The local minimax point $z^{\ast}$ was initialized as $r(-5,5)$, and the optimal dual variables $u_{i}$, for $i \in \alpha^{\ast}$, were chosen as $r(-30,0)$. The constraint values were set as $s_{i} = 0$ for $i \in \alpha^{\ast}$ and $s_{j} = r(-1,0)$ for $j \notin \alpha^{\ast}$. The two types of problems differ primarily in the generation of $G_{11}$ and $D$. All Type~1 problems satisfy Assumption~\ref{assumption}, whereas all Type~2 problems satisfy Assumption~\ref{assumption2}.
	
	For Type~1 problems, all elements of the matrix $D$ was set to $r(-1,1)$, and the rows of $D$ were normalized to unit length. The matrix $B_{\alpha^{\ast}}$ was constructed to be of full row rank. Then,
	\begin{align*}
		G_{11} = \Gamma_{\emptyset} + 
		\begin{bmatrix}
			G_{12} & A_{\alpha^{\ast}}^{T}
		\end{bmatrix}
		\begin{bmatrix}
			G_{22} & B_{\alpha^{\ast}}^{T} \\
			B_{\alpha^{\ast}} & 0
		\end{bmatrix}^{-1}
		\begin{bmatrix}
			G_{12}^{T} \\ A_{\alpha^{\ast}}
		\end{bmatrix}.
	\end{align*}
	Finally, we set $h = s - Dz^{\ast}$ and $c = -D^{T}u - Gz^{\ast}$.
	
	For Type~2 problems, $G_{11}$ was computed as
	\begin{align*}
		G_{11} = \Gamma_{\emptyset} + G_{12}G_{22}^{-1}G_{12}^{T}.
	\end{align*}
	Let $Q$ denote the eigenvector matrix from the eigendecomposition of $G$. The auxiliary matrix $D_{t}$ was generated with entries $r(-1,1)$, and $D$ was obtained as $D = Q D_{t}$. Then, we set $h = s - Dz^{\ast}$ and $c = -D^{T}u - Gz^{\ast}$.

	The dual algorithm always adds the most violated constraint to the active set. When solving Type~1 problems, if the algorithm detects an infeasible subproblem, the instance is discarded since it does not satisfy Assumption~\ref{assumption}. Note that verifying Assumption~\ref{assumption} is generally difficult, and therefore, retaining an instance does not guarantee that it satisfies this assumption. All Type~2 problems inherently satisfy Assumption~\ref{assumption2}.
	
	All problems of the same scale were repeated 20 times, and the average value was used to represent the performance of that scale (except for numerical error). We denote by $\hat{z}$ and $\hat{u}$ the primal and dual solutions obtained by the proposed algorithm, respectively. The reported errors $\|z^{\ast} - \hat{z}\|_{2}$ and $\|u^{\ast} - \hat{u}\|_{2}$ correspond to the largest deviations among 20 instances of the same scale. All computations were performed using \textsc{MATLAB} R2023a on a PC equipped with an Intel(R) Core(TM) i9-12900K 3.20~GHz CPU and 128~GB RAM.
 
	The numerical results are summarized in Table~\ref{table1} and Figures~\ref{figure1},\ref{figure2}. Table~\ref{table1} presents the performance of the dual algorithm for problems of different scales. All tested problems were of Type~2; the reason for this choice is discussed in Figure~\ref{figure2}. The column “Number of basic changes” reports how many times the active set was modified, where the first number indicates additions to the active set and the number in parentheses indicates deletions. We also report the total number of arithmetic operations (multiplications + divisions + $10\times$ square roots) required by the algorithm. As shown in Table~\ref{table1}, both $\|z^{\ast} - \hat{z}\|_{2}$ and $\|u^{\ast} - \hat{u}\|_{2}$ remain below ${\rm 1e-9}$ even for large scale problems, confirming that our implementation of the dual algorithm is numerically stable. Notably, for Problem~6, the total runtime is approximately 2.5 hours and the total number of operations is about ${\rm 5e+16}$, which is acceptable comparing with the dual active set method for quadratic programming problems of similar size.
	
	\begin{longtable}{>{\centering\arraybackslash}p{1.5cm}  
			>{\centering\arraybackslash}p{0.6cm}  
			>{\centering\arraybackslash}p{0.6cm}  
			>{\centering\arraybackslash}p{0.6cm}  
			>{\centering\arraybackslash}p{0.6cm}  
			>{\centering\arraybackslash}p{2.2cm}  
			>{\centering\arraybackslash}p{1.8cm}  
			>{\centering\arraybackslash}p{1.8cm}  
			>{\centering\arraybackslash}p{1.8cm}  
			>{\centering\arraybackslash}p{1.8cm}  
		}
		\caption{The computing results for the problems of Type2}
		\label{table1}\\
		\toprule
		\begin{tabular}{l}
			Problem\\
			index
		\end{tabular} & \multicolumn{4}{l}{Problem scale} &\multicolumn{5}{l}{Computing results}\\
		\cmidrule(lr){2-5} \cmidrule(lr){6-10}
		& $n_{x}$ & $n_{y}$ & $m$ & $n_{a}$ & number of basis changes & time(s) & operations(in $10^{8}$s) & $\|z^{\ast} - \hat{z}\|_{2}$ & $\|u^{\ast} - \hat{u}\|_{2}$\\
		\midrule
		1 & 100 & 200 & 300 & 100 & 105.5(5.5) & 0.0065 & 6 & 5.923e-14 & 5.586e-12\\
		2 & 1000 & 500 & 1500 & 250 & 299.65(49.65) & 0.2945 & 1321 & 2.314e-13 & 2.177e-12\\
		3 & 1000 & 1000 & 2000 & 500 & 539.25(39.25) & 0.9421 & 6624 & 3.026e-13 & 5.041e-12\\
		4 & 1000 & 2000 & 3000 & 1000 & 1043.6(43.6) & 4.8299 & 52822 & 4.455e-13 & 1.242e-11\\
		5 & 5000 & 10000 & 15000 & 5000 & 5234.05(234.05) & 902.9232 & 33056585 & 2.502e-12 & 1.692e-10 \\
		6 & 10000 & 20000 & 30000 & 10000 & 10450.5(450.5) & 9260.9908 & 525168646 & 4.993e-12 & 4.721e-10 \\
		\bottomrule
	\end{longtable}

	Figure~\ref{figure1} reports the computational time per iteration for a representative Type~2 problem with dimensions $n_{x} = 1000$, $n_{y} = 2000$, $m = 3000$, and $n_{a} = 1000$. As shown in Figure~\ref{figure1(a)}, iterations involving constraint deletions require significantly more time than those without any deletion. For this problem, the longest iteration that includes a constraint deletion is approximately $72$ times slower than the fastest iteration without deletion, and this ratio tends to increase with problem size. The difference primarily arises because iterations without constraint deletions involve only one full step, whereas iterations with deletions contain both one full step and at least one partial step. Figure~\ref{figure1(b)} compares the computational time of full and partial steps. It is evident that partial steps consume substantially more time than full steps, leading to similar trends observed in Figures~\ref{figure1(a)} and~\ref{figure1(b)}. The increased cost of partial steps is mainly due to the Givens rotations required when updating the active set. For an iteration with active set $\alpha$ containing $|\alpha|$ elements, the computational complexity of adding a constraint is approximately $\mathcal{O}(|\alpha|^{2} + |\alpha|n)$, where $n = n_{x} + n_{y}$. In contrast, performing Givens rotations during a constraint deletion has complexity $\mathcal{O}((|\alpha| - k)^{2})$, where $k$ denotes the position of the constraint to be removed in the active set. Since partial steps require these Givens rotations whereas full steps do not, the computational behavior in Figures~\ref{figure1(a)} and~\ref{figure1(b)} closely follows the same pattern.

	\begin{figure}
		\centering
		\subfigure[The time of iteration number of S-pair]{
			\label{figure1(a)}
			\includegraphics[scale=0.4]{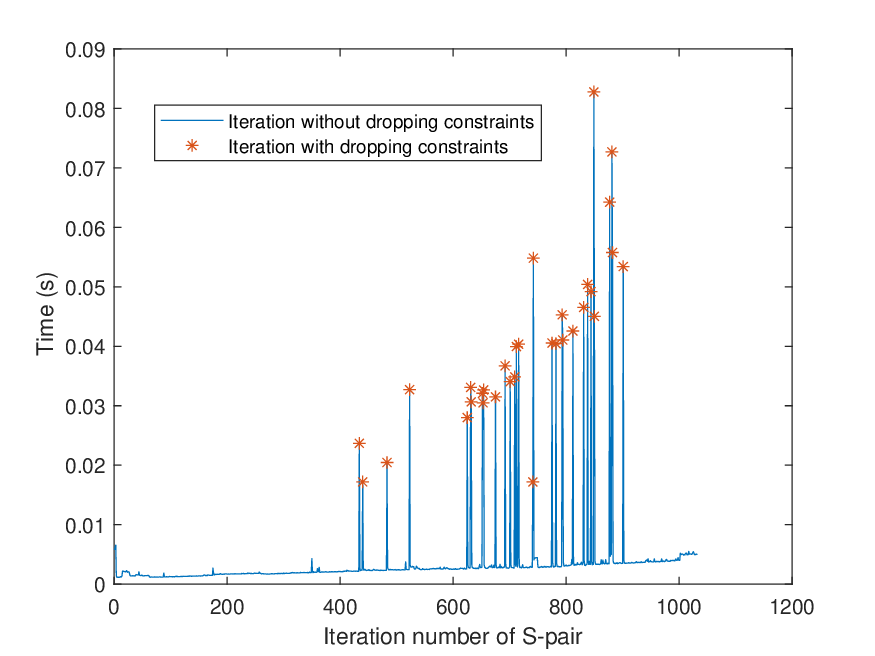}
		}
		\subfigure[The time of full (partial) step]{
			\label{figure1(b)}
			\includegraphics[scale=0.4]{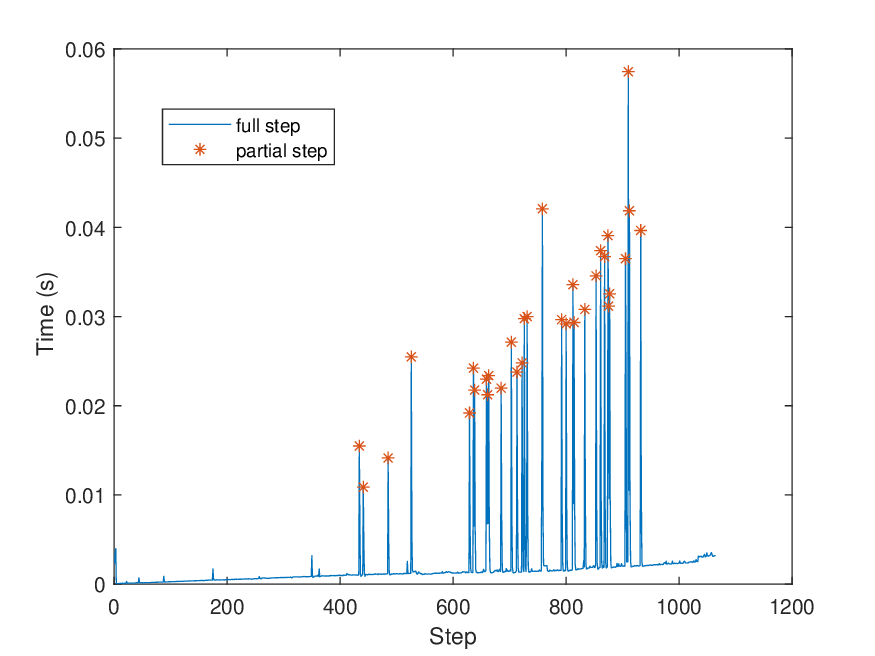}
		}
		
		\caption{Iteration number of S-pair and step time for a given problem.}
		\label{figure1}
	\end{figure}

	Figure~\ref{figure2} compares the number of constraint deletions and the total solution time for Type~1 and Type~2 problems as the problem scale increases. As shown in Figure~\ref{figure2(a)}, the number of constraint deletions grows approximately linearly with the problem scale. In Figure~\ref{figure2(b)}, the total computational time for both problem types increases between $\mathcal{O}(i^{3})$ and $\mathcal{O}(i^{4})$, where $n_{x} = i$, $n_{y} = 2i$, $m = 3i$, and $n_{a} = i$. It is observed that both the number of deletions and the total solution time of Type~1 problems are significantly smaller than those of Type~2 problems. This difference primarily arises from the way the two problem types are generated. All randomly generated Type~2 problems satisfy Assumption~\ref{assumption2}, meaning that all experimental results are retained. In contrast, some randomly generated Type~1 problems violate Assumption~\ref{assumption}, leading to the exclusion of infeasible cases from the reported results. Extensive experiments indicate that it is difficult to construct a Type~1 problem whose performance is worse than that of a Type~2 problem of the same scale. This observation also explains why only Type~2 problems are considered in Table~\ref{table1}.
	
	\begin{figure}
		\centering
		\subfigure[Trend of constraint deletion count]{
			\label{figure2(a)}
			\includegraphics[scale=0.4]{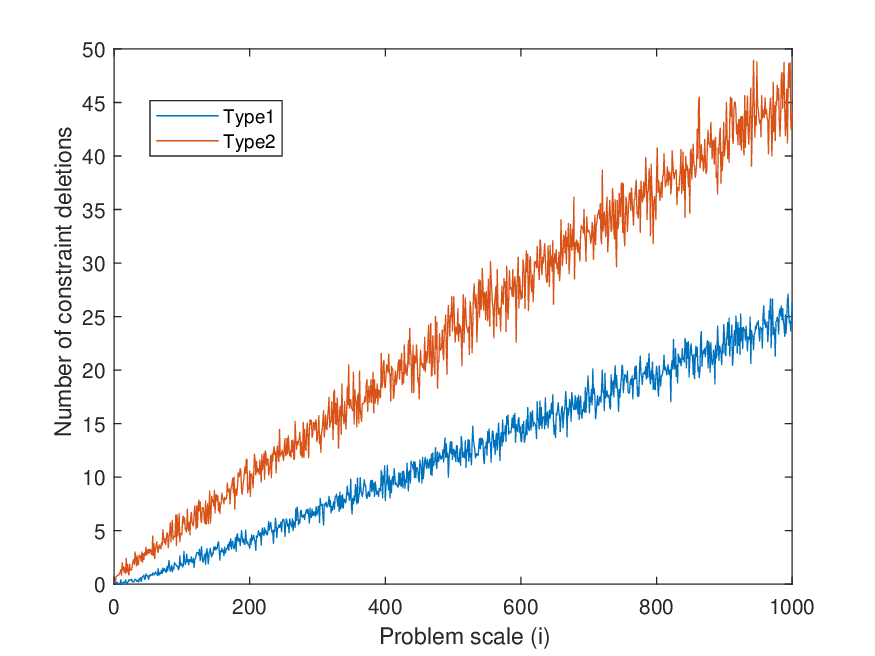}
		}
		\subfigure[Trend of computation time]{
			\label{figure2(b)}
			\includegraphics[scale=0.4]{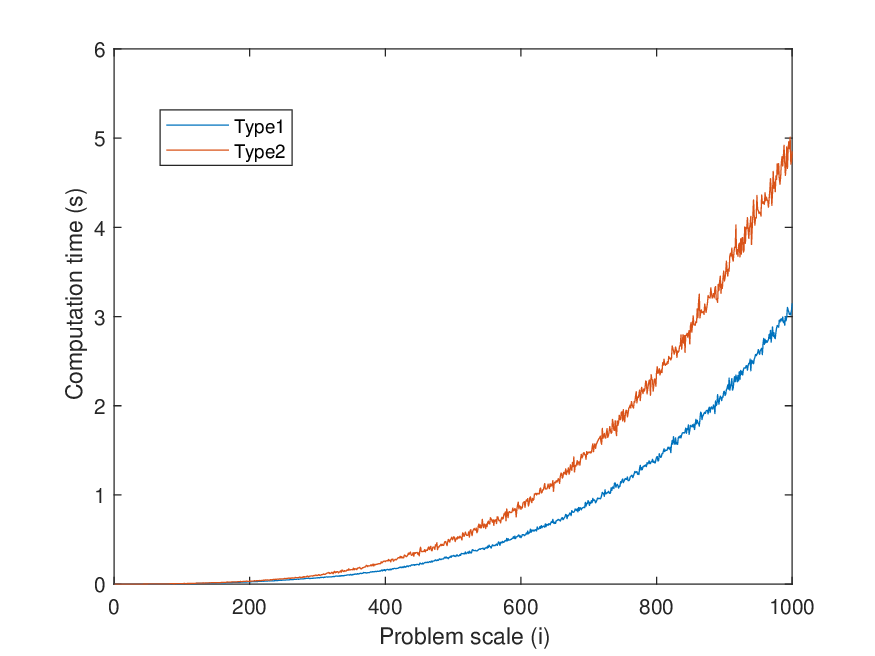}
		}
		
		\caption{The trend of constraints deletion count and computation time for the problems of two types with the scale $n_{x} = i, n_{y} = 2i, m = 3i, n_{a} = i$.}
		\label{figure2}
	\end{figure}
	
	\subsection{Adversarial attack on a mean-covariance portfolio model}
	In this subsection we evaluate the proposed local minimax point and the dual algorithm~\ref{dual algorithm} by constructing adversarial attacks against a mean-covariance portfolio model~\cite{markowitzPortfolioSelection1952,mertonAnalyticDerivationEfficient1972}. The attacker (adversary) chooses $x\in\mathbb{R}^{n}$ to degrade the investor's performance, while the investor chooses $y\in\mathbb{R}^{n}$ to maximize return adjusted for risk. The adversarial problem is formulated as
	\begin{align}\label{eq:attack}
		\begin{aligned}
			\min_{x\in\mathbb{R}^{n}}\;\max_{y\in\mathbb{R}^{n}}\quad & f(x,y)
			= \frac{1}{2}x^{T}Hx + x^{T}W y - \frac{1}{2}y^{T}\Sigma_{y}y + \mu^{T}y,\\
			\text{s.t.}\quad & s_{i}(x,y) = x_{i} + y_{i} - (12 - b) \leq 0,\quad i=1,\dots,n,
		\end{aligned}
	\end{align}
	where $n$ denotes the number of assets. The objective function $f(x,y)$ captures a simple mean-covariance trade-off for the investor $-\frac{1}{2} y^{T}\Sigma_{y}y + \mu^{T}y$, while the adversary implements an attack via the vector $x$. The adversary incurs a quadratic cost $\frac{1}{2}x^{T}Hx$ and interacts with the investor via $x^{T}Wy$, which models how adversarial actions affect the investor's decisions. The constraint $x_{i}+y_{i}\le 12-b$ models a limited liquidity budget: a larger $b$ yields a tighter feasible set, i.e. a smaller value of $12-b$, which reflects deteriorated market liquidity. The parameter $b$ is referred to as the liquidity intensity parameter.
	
	We use the constituents of the NASDAQ-100 index for the year 2024 (252 trading days), hence $n=100$. For each asset $i$, we collect its daily closing prices $p_{t, i}$ and daily trading volumes over the period from 2024-01-01 to 2024-12-31, where $p_{t,i}$ denotes the closing price of asset $i$ on trading day $t$. Based on the price series, we compute the daily arithmetic returns as $R_{t,i}=(p_{t,i}-p_{t-1,i})/p_{t-1,i}$ and define
	\begin{align*}
		\mu := \frac{1}{T - 1}\sum_{t=2}^{T} R_{t}, \qquad
		\Sigma_{y} := \frac{1}{T-1}\sum_{t=2}^{T}(R_{t}-\mu)(R_{t}-\mu)^{T},
	\end{align*}
	with $T=252$. Let $\mathrm{ADV}_{i}$ denote the average daily trading volume of asset $i$. The liquidity weighting matrix is
	\begin{align*}
		W := \operatorname{diag}\big(1/\mathrm{ADV}_{1},\dots,1/\mathrm{ADV}_{n}\big).
	\end{align*}
	The attacker matrix is chosen as $H = I_{n}$ (the $n\times n$ identity). To make the three matrices $H$, $W$ and $\Sigma_{y}$ comparable in magnitude we normalize them using $M_{\mu}:=\|\mu\|_{1}$. Specifically we apply the scalings
	\begin{align*}
		H \leftarrow\frac{0.1M_{\mu}}{\|H\|_{1}} \cdot H,\qquad
		W \leftarrow \frac{M_{\mu}}{\|W\|_{1}} \cdot W,\qquad
		\Sigma_{y}\leftarrow \frac{0.2M_{\mu}}{\|\Sigma_{y}\|_{1}} \cdot \Sigma_{y},
	\end{align*}
	so that $\|H\|_{1}=0.1M_{\mu}$, $\|W\|_{1}=M_{\mu}$ and $\|\Sigma_{y}\|_{1}=0.2M_{\mu}$.
	
	We compare the proposed minimax attack (computed by the dual algorithm) with three baselines:
	\begin{enumerate}
		\item \textbf{Random attack.}  An attack vector $x$ is generated at random. The attack $x$ is drawn from the uniform distribution $r(b-12, 12-b)$ where $r(a, b)$ denotes a uniformly distributed random number between $a$ and $b$.
		
		\item \textbf{No-long (forbid long positions) attack.} Identify the 20 assets with the largest historical average daily return $\mu_{i}$. For each such asset set $x_{i} = 12-b$, which forces $y_{i}\le 0$ by the constraint $x_{i}+y_{i}\le 12-b$ and thus prevents the investor $y$ from taking long positions in these high-return assets. For all other assets set $x_{i}=0$.
		
		\item \textbf{NI attack.} We reformulate the two-player zero-sum game using a Nikaido–Isoda function~\cite{vonheusingerOptimizationReformulationsGeneralized2009}, which is a widely used approach in the analysis of generalized Nash equilibrium problems. Let $z=(x,y)$, $\hat z=(\hat x,\hat y)$ and $\gamma > 0$, define Nikaido-Isoda function
		\begin{align*}
			\Psi_{\gamma}(z,\hat z)=f(x,\hat y)-f(\hat x,y) - \frac{\gamma}{2}\|z - \hat{z}\|^{2},
		\end{align*}
		and the associated value function
		\begin{align*}
			\hat{V}_{\gamma}(z)=\max_{s(\hat z)\le 0}\left\lbrace\,\Psi(z,\hat z)-\frac{\gamma}{2}\|z-\hat z\|^{2}\right\rbrace.
		\end{align*}
		The NI attack is then obtained by minimizing $\hat{V}_{\gamma}(z)$ with respect to $z$.
	\end{enumerate}
	
	Let $x_{att}$ denote the adversary's decisions after an attack and $x_{be} = 0$ denote the no attack baseline. For a given attack vector $x$, we solve the maximization problem
	\begin{align*}
		q(x):=\max_{y \in Y(x)} f(x,y),\quad Y(x) = \left\lbrace y \mid x + y \leq 12-b \right\rbrace,
	\end{align*}
	which ensures the investor's optimal response under the imposed attack. We measure the effectiveness of an attack by the relative reduction:
	\begin{align*}
		\rho := \frac{q(x_{be})-q(x_{att})}{|q(x_{be})|}.
	\end{align*}
	The larger values of $\rho$ indicate a more effective attack relative to the no attack baseline. For cases in which the attack fails to reduce the objective, $\rho$ is set to zero.
	
	The random attack was executed over $2000$ independent trials. The trial achieving the largest value of $\rho$ was used for evaluation. The NI attack minimizes $\hat{V}_{\gamma}(z)$ with respect to $z$ using a Gauss–Newton method~\cite{nocedalNumericalOptimization2006}, with the parameter set to $\gamma = 0.01$. For all attack methods, the maximization problems $q(x)$ were solved using the \texttt{quadprog} function in \textsc{MATLAB}. All computations were performed using \textsc{MATLAB} R2023a on a PC equipped with an Intel(R) Core(TM) i9-12900K 3.20~GHz CPU and 128~GB RAM.
	
	The experimental results are reported in Figure~\ref{figure3}. Compared with the three baseline methods, the proposed minimax attack demonstrates clear effectiveness. Extensive numerical experiments reveal that the random attack generally fails to produce noticeable impact. Both the random and no-long attacks are consistently ineffective, primarily because the adversary incurs a quadratic cost term $\frac{1}{2}x^{T}Hx$, which constrains the set of admissible strategies. The NI attack remains effective but exhibits inferior performance relative to the proposed minimax attack. When $0\leq b \leq \tilde{b}$, the solution $z^{\ast}$ of the unconstrained problem~\eqref{eq:attack} satisfies $s(z^{\ast}) \leq 0$. In this regime, the minimax attack and NI attack exhibit similar levels of relative reduction $\rho$. When $\tilde{b} < b \leq 12$, both the attacker $x$ and the investor $y$ are influenced by the liquidity constraints. In this range, the relative reduction $\rho$ of the minimax attack first decreases and then increases. This behavior is explained by Figure~\ref{figure3(b)}, which depicts the expected return $q$ (i.e., the objective function value) before and after the attack. The non-monotonic trend in $\rho$ arises because the rate of decline in $q$ differs between the before and after attack, causing $\rho$ to initially decrease and subsequently increase. As liquidity becomes nearly exhausted ($b$ approaching $12$), the minimax attack exhibits pronounced effectiveness. Moreover, the experiments indicate that as market liquidity deteriorates corresponding to smaller values of $(12-b)$, a larger number of variables enter the active set, which is consistent with empirical observations in real markets.
	
	\begin{figure}
		\centering
		\subfigure[Relative reduction $\rho$ with respect to liquidity intensity parameter $b$.]{
			\label{figure3(a)}
			\includegraphics[scale=0.4]{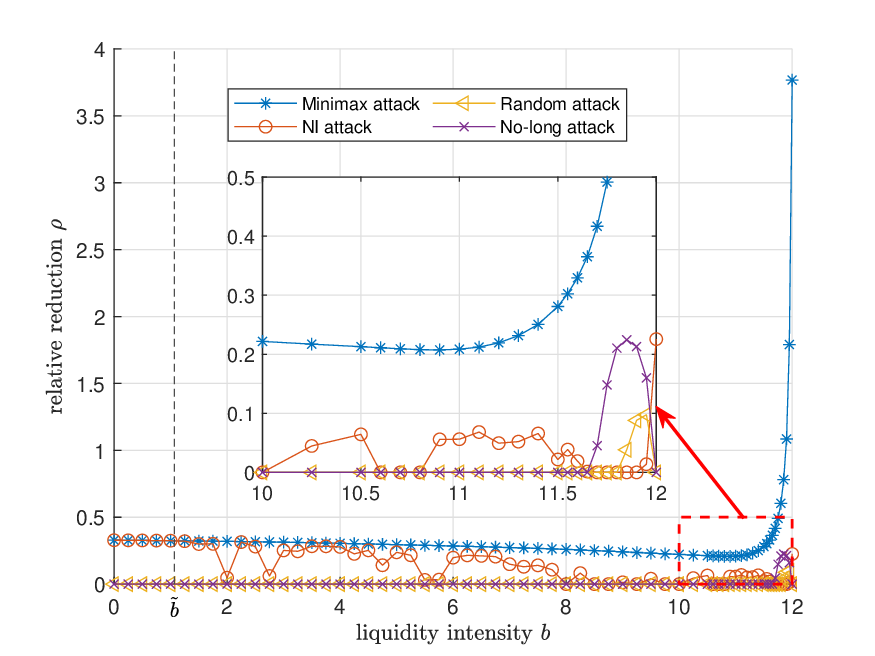}
		}
		\subfigure[Expected return $q$ with respect to liquidity intensity parameter $b$.]{
			\label{figure3(b)}
			\includegraphics[scale=0.4]{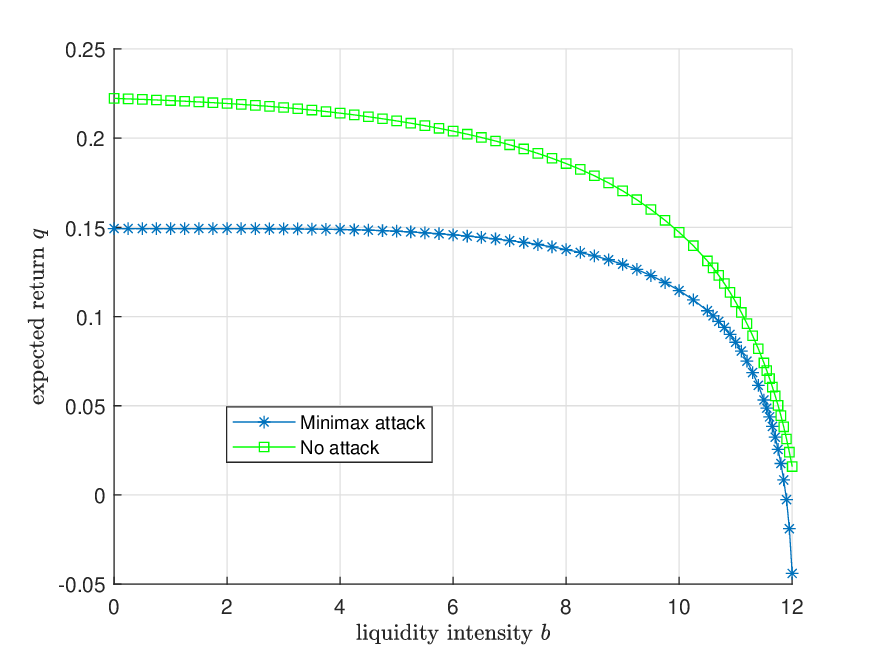}
		}
		
		\caption{Relative reduction $\rho$ and expected return $q$ with respect to liquidity intensity parameter $b$.}
		\label{figure3}
	\end{figure}
	
	\subsection{Illustrative examples}
	To demonstrate the different components of the dual algorithm~\ref{dual algorithm}, we apply it to two representative test problems. 
	Example~1 satisfies Assumption~\ref{assumption2}, that is, the matrix $DG^{-1}D^{T}$ is negative semi-definite. 
	In contrast, Example~2 satisfies Assumption~\ref{assumption} but $DG^{-1}D^{T}$ is not negative semi-definite.

	Example 1:
	\begin{align}
		\left\lbrace
		\begin{aligned}
			\min_{x\in\mathbb{R}^{n_{x}}}\max_{y\in\mathbb{R}^{n_{y}}}\quad& f(x, y) = \frac{1}{2}x^{T}G_{11}x + x^{T}G_{12}y + \frac{1}{2}y^{T}G_{22}y + c_{x}^{T}x + c_{y}^{T}y,\\
			{\rm s.t.}\quad & s(x, y) = Ax + By + h \leq 0.
		\end{aligned}\right.
	\end{align}
	where
	\begin{align*}
		G_{11} & = \begin{bmatrix}
			-1 & -1\\
			-1 & 0
		\end{bmatrix}, \quad G_{12} = \begin{bmatrix}
			2 &	-1 & 0 & 0\\
			1 &	0 &	-1 & 1
		\end{bmatrix}, \quad G_{22} = \begin{bmatrix}
			-3 & 1 & 0 & 1\\
			1 & -1 & 1 & -1\\
			0 & 1 & -2 & 2\\
			1 & -1 & 2 & -4
		\end{bmatrix},\\
		A & = \begin{bmatrix}
			0 & -1\\
			1 & 0\\
			1 & 1\\
			0 & 0\\
			1 & 0
		\end{bmatrix},\quad B = \begin{bmatrix}
			0 & 1 & -2 & 1\\
			2 & 0 & 0 & 1\\
			2 & -1 & 2 & 0\\
			0 & 0 & 0 & 1\\
			2 & 0 & 0 & -2
		\end{bmatrix},\\
		c_{x}^{T} & = \begin{bmatrix}
			6 & 6
		\end{bmatrix}, \quad c_{y}^{T} = \begin{bmatrix}
		0 & 1 & 4 & -4
		\end{bmatrix},\quad h^{T} = \begin{bmatrix}
			-4 & -1 & 2 & -1 & -8
		\end{bmatrix}.
	\end{align*}
	It is easy to check
	\begin{align*}
		DG^{-1}D^{T} = \begin{bmatrix}
			-3 & 3 & 6 & 1 & 0\\
			3 & -12 & -15 & -3 & -3\\
			6 & -15 & -21 & -4 & -3\\
			1 & -3 & -4 & -1 & 0\\
			0 & -3 & -3 & 0 & -3
		\end{bmatrix} \preccurlyeq 0.
	\end{align*}
	We have 
	\begin{align*}
	z_{0}^{T} = -c^{T}G^{-1} = 
	\begin{bmatrix}
		0 & -1 & 12 & 31 & 24 & 6
	\end{bmatrix},
	\end{align*}
	and thus obtain the initial S-pair $(z_{0}, \emptyset)$. 
	Table~\ref{solution path1} presents several representative solution paths for Example~1. 
	Solution paths~1–3 illustrate different possible trajectories from the starting point $z_{0}$ to the optimal solution $z^{\ast}$. 
	Since the violated constraint selected in the first iteration differs among the three paths, their step directions in iteration~1 are also distinct. 
	Among them, solution path~2 requires the fewest iterations, as its active set is $\{3\}$—the violated constraint chosen in the first iteration. 
	In solution path~3, we observe that $\|d\| = 0$ while $t > 0$ in iteration~3, indicating that a nonzero step is taken in the dual space, whereas no movement occurs in the primal space. 
	For all possible solution paths, it can be readily verified that the objective function value decreases in every iteration and strictly decreases in each S-pair iteration.

	\begin{longtable}{>{\centering\arraybackslash}p{1.2cm}  
			>{\centering\arraybackslash}p{1.2cm}  
			>{\centering\arraybackslash}p{1.2cm}    
			>{\centering\arraybackslash}p{0.3cm}  
			>{\centering\arraybackslash}p{0.8cm}  
			>{\centering\arraybackslash}p{1.1cm}  
			>{\centering\arraybackslash}p{0.3cm}  
			>{\centering\arraybackslash}p{1.1cm}  
			>{\centering\arraybackslash}p{0.8cm}  
			>{\centering\arraybackslash}p{0.3cm}  
			>{\centering\arraybackslash}p{0.3cm}  
			>{\centering\arraybackslash}p{0.3cm}  
			>{\raggedright\arraybackslash}p{3.6cm} 
		}
		\caption{Solution paths of the dual algorithm for Example 1}\label{solution path1}\\
		\toprule
		Iteration & \multicolumn{5}{l}{At the start of iteration} &\multicolumn{6}{l}{Computed during iteration} & Remarks \\
		\cmidrule(lr){2-6} \cmidrule(lr){7-12}
		& $z$ & $s$ & $f$ & $\alpha$ & $u$ & $p$ & $d$ & $r$ & $t_{1}$ & $t_{2}$ & $k$ &  \\
		\midrule
		\multicolumn{13}{l}{Solution path 1: $z_{0} \rightarrow z_{4} \rightarrow z_{5} \rightarrow z^{\ast}$.}\\
		\midrule
		1 & $\begin{bmatrix}
			0\\
			-1\\
			12\\
			31\\
			24\\
			6
		\end{bmatrix}$ & $\begin{bmatrix}
			-14\\
			29\\
			42\\
			5\\
			4
		\end{bmatrix}$ & $\frac{97}{2}$ & $\emptyset$ & $\begin{bmatrix}
			0\\
			0\\
			0\\
			0\\
			0
		\end{bmatrix}$ & $2$ & $\begin{bmatrix}
			1\\
			0\\
			-5\\
			-12\\
			-9\\
			-3
		\end{bmatrix}$ & $-$ & $\infty$ & $\frac{29}{12}$ & $-$ & \begin{tabular}{l}
			full step:\\
			add constraint $2$
		\end{tabular}\\
		2 & $\begin{bmatrix}
		 \frac{29}{12}\\  
			-1\\
			-\frac{1}{12}\\
			2\\
			\frac{9}{4}\\
			-\frac{5}{4}
		\end{bmatrix}$ & $\begin{bmatrix}
			-\frac{27}{4}\\
			0\\
			\frac{23}{4}\\
			-\frac{9}{4}\\
			-\frac{13}{4}
		\end{bmatrix}$ & $\frac{323}{24}$ & $\left\lbrace 2\right\rbrace$ & $\begin{bmatrix}
			0\\
			-\frac{29}{12}\\
			0\\
			0\\
			0
		\end{bmatrix}$ & $3$ & $\begin{bmatrix}
			-\frac{1}{4}\\
			0\\
			\frac{1}{4}\\
			1\\
			-\frac{3}{4}\\
			-\frac{1}{4}
		\end{bmatrix}$ & $\begin{bmatrix}
			-\frac{5}{4}
		\end{bmatrix}$ & $\frac{29}{15}$ & $\frac{23}{9}$ & $2$ & \begin{tabular}{l}
			partial step:\\
			drop constraint $2$
		\end{tabular}\\
		3 & $\begin{bmatrix}
			\frac{29}{15}\\
			-1\\
			\frac{2}{5}\\
			\frac{59}{15}\\
			\frac{4}{5}\\
			-\frac{26}{15}
		\end{bmatrix}$ & $\begin{bmatrix}
			-\frac{12}{5}\\
			0\\
			\frac{7}{5}\\
			-\frac{41}{15}\\
			-\frac{9}{5}
		\end{bmatrix}$ & $\frac{694}{75}$ & $\emptyset$ & $\begin{bmatrix}
			0\\
			0\\
			-\frac{29}{15}\\
			0\\
			0
		\end{bmatrix}$ & $3$ & $\begin{bmatrix}
			1\\
			0\\
			-6\\
			-14\\
			-12\\
			-4
		\end{bmatrix}$ & $-$ & $\infty$ & $\frac{1}{15}$ & $-$ & \begin{tabular}{l}
			full step:\\
			add constraint $3$
		\end{tabular}\\
		4 & $\begin{bmatrix}
			2\\
			-1\\
			0\\
			3\\
			0\\
			-2
		\end{bmatrix}$ & $\begin{bmatrix}
			-2\\
			-1\\
			0\\
			-3\\
			-2
		\end{bmatrix}$ & $\frac{13}{2}$ & $\left\lbrace 3\right\rbrace$ & $\begin{bmatrix}
			0\\
			0\\
			-2\\
			0\\
			0
		\end{bmatrix}$ & & & & & & & \begin{tabular}{l}
			stop:\\
			all constraints satisfied
		\end{tabular}\\
		\midrule
		\multicolumn{13}{l}{Solution path 2: $z_{0} \rightarrow z^{\ast}$.}\\
		\midrule
		1 & $\begin{bmatrix}
			0\\
			-1\\
			12\\
			31\\
			24\\
			6
		\end{bmatrix}$ & $\begin{bmatrix}
			-14\\
			29\\
			42\\
			5\\
			4
		\end{bmatrix}$ & $\frac{97}{2}$ & $\emptyset$ & $\begin{bmatrix}
			0\\
			0\\
			0\\
			0\\
			0
		\end{bmatrix}$ & $3$ & $\begin{bmatrix}
			1\\
			0\\
			-6\\
			-14\\
			-12\\
			-4
		\end{bmatrix}$ & $-$ & $\infty$ & $2$ & $-$ & \begin{tabular}{l}
			full step:\\
			add constraint $3$
		\end{tabular}\\
		2 & $\begin{bmatrix}
			2\\
			-1\\
			0\\
			3\\
			0\\
			-2
		\end{bmatrix}$ & $\begin{bmatrix}
			-2\\
			-1\\
			0\\
			-3\\
			-2
		\end{bmatrix}$ & $\frac{13}{2}$ & $\left\lbrace 3\right\rbrace$ & $\begin{bmatrix}
			0\\
			0\\
			-2\\
			0\\
			0
		\end{bmatrix}$ & & & & & & & \begin{tabular}{l}
			stop:\\
			all constraints satisfied
		\end{tabular}\\
		\midrule
		\multicolumn{13}{l}{Solution path 3: $z_{0} \rightarrow z_{1} \rightarrow z_{2} \rightarrow z_{2} \rightarrow z_{3} \rightarrow z_{4} \rightarrow z_{5} \rightarrow z^{\ast}$.}\\
		\midrule
		1 & $\begin{bmatrix}
			0\\
			-1\\
			12\\
			31\\
			24\\
			6
		\end{bmatrix}$ & $\begin{bmatrix}
			-14\\
			29\\
			42\\
			5\\
			4
		\end{bmatrix}$ & $\frac{97}{2}$ & $\emptyset$ & $\begin{bmatrix}
			0\\
			0\\
			0\\
			0\\
			0
		\end{bmatrix}$ & $4$ & $\begin{bmatrix}
			0\\
			0\\
			-1\\
			-2\\
			-2\\
			-1
		\end{bmatrix}$ & $-$ & $\infty$ & $5$ & $-$ & \begin{tabular}{l}
			full step:\\
			add constraint $4$
		\end{tabular}\\
		2 & $\begin{bmatrix}
			0\\
			-1\\
			7\\
			21\\
			14\\
			1
		\end{bmatrix}$ & $\begin{bmatrix}
			-9\\
			14\\
			22\\
			0\\
			4
		\end{bmatrix}$ & $36$ & $\left\lbrace 4\right\rbrace$ & $\begin{bmatrix}
			0\\
			0\\
			0\\
			-5\\
			0
		\end{bmatrix}$ & $5$ & $\begin{bmatrix}
			1\\
			0\\
			-2\\
			-6\\
			-3\\
			0
		\end{bmatrix}$ & $\begin{bmatrix}
			0
		\end{bmatrix}$ & $\infty$ & $\frac{4}{3}$ & $-$ & \begin{tabular}{l}
			full step:\\
			add constraint $5$
		\end{tabular}\\
		3 & $\begin{bmatrix}
			\frac{4}{3}\\
			-1\\
			\frac{13}{3}\\
			13\\
			10\\
			1
		\end{bmatrix}$ & $\begin{bmatrix}
			-9\\
			10\\
			18\\
			0\\
			0
		\end{bmatrix}$ & $\frac{100}{3}$ & $\left\lbrace 4, 5\right\rbrace $ & $\begin{bmatrix}
			0\\
			0\\
			0\\
			-5\\
			-\frac{4}{3}
		\end{bmatrix}$ & $2$ & $\begin{bmatrix}
			0\\
			0\\
			0\\
			0\\
			0\\
			0
		\end{bmatrix}$ & $\begin{bmatrix}
			-3\\ -1
		\end{bmatrix}$ & $\frac{4}{3}$ & $\infty$ & $5$ & \begin{tabular}{l}
			partial step:\\
			drop constraint $5$
		\end{tabular}\\
		4 & $\begin{bmatrix}
			\frac{4}{3}\\  
			-1\\
			\frac{13}{3}\\
			13\\
			10\\
			1
		\end{bmatrix}$ & $\begin{bmatrix}
			-9\\ 10\\ 18\\ 0\\ 0
		\end{bmatrix}$ & $\frac{100}{3}$ & $\left\lbrace 4\right\rbrace$ & $\begin{bmatrix}
			0\\
			-\frac{4}{3}\\
			0\\
			-1\\
			0
		\end{bmatrix}$ & $2$ & $\begin{bmatrix}
			1\\ 0\\ -2\\ -6\\ -3\\ 0
		\end{bmatrix}$ & $\begin{bmatrix}
			-3
		\end{bmatrix}$ & $\frac{1}{3}$ & $\frac{10}{3}$ & $4$ & \begin{tabular}{l}
			partial step:\\
			drop constraint $4$
		\end{tabular}\\
		5 & $\begin{bmatrix}
			\frac{5}{3}\\ -1\\ \frac{11}{3}\\ 11\\ 9\\ 1
		\end{bmatrix}$ & $\begin{bmatrix}
			-9\\ 9\\ 17\\ 0\\ -1
		\end{bmatrix}$ & $\frac{191}{6}$ & $\emptyset$ & $\begin{bmatrix}
			0\\
			-\frac{5}{3}\\
			0\\
			0\\
			0
		\end{bmatrix}$ & $2$ & $\begin{bmatrix}
			1\\ 0\\ -5\\ -12\\ -9\\ -3
		\end{bmatrix}$ & $-$ & $\infty$ & $\frac{3}{4}$ & $-$ & \begin{tabular}{l}
			full step:\\
			add constraint $2$
		\end{tabular}\\
		6 & $\begin{bmatrix}
			\frac{29}{12}\\  
			-1\\
			-\frac{1}{12}\\
			2\\
			\frac{9}{4}\\
			-\frac{5}{4}
		\end{bmatrix}$ & $\begin{bmatrix}
			-\frac{27}{4}\\
			0\\
			\frac{23}{4}\\
			-\frac{9}{4}\\
			-\frac{13}{4}
		\end{bmatrix}$ & $\frac{323}{24}$ & $\left\lbrace 2\right\rbrace$ & $\begin{bmatrix}
			0\\
			-\frac{29}{12}\\
			0\\
			0\\
			0
		\end{bmatrix}$ & $3$ & $\begin{bmatrix}
			-\frac{1}{4}\\
			0\\
			\frac{1}{4}\\
			1\\
			-\frac{3}{4}\\
			-\frac{1}{4}
		\end{bmatrix}$ & $\begin{bmatrix}
			-\frac{5}{4}
		\end{bmatrix}$ & $\frac{29}{15}$ & $\frac{23}{9}$ & $2$ & \begin{tabular}{l}
			partial step:\\
			drop constraint $2$
		\end{tabular}\\
		7 & $\begin{bmatrix}
			\frac{29}{15}\\
			-1\\
			\frac{2}{5}\\
			\frac{59}{15}\\
			\frac{4}{5}\\
			-\frac{26}{15}
		\end{bmatrix}$ & $\begin{bmatrix}
			-\frac{12}{5}\\
			0\\
			\frac{7}{5}\\
			-\frac{41}{15}\\
			-\frac{9}{5}
		\end{bmatrix}$ & $\frac{694}{75}$ & $\emptyset$ & $\begin{bmatrix}
			0\\
			0\\
			-\frac{29}{15}\\
			0\\
			0
		\end{bmatrix}$ & $3$ & $\begin{bmatrix}
			1\\
			0\\
			-6\\
			-14\\
			-12\\
			-4
		\end{bmatrix}$ & $-$ & $\infty$ & $\frac{1}{15}$ & $-$ & \begin{tabular}{l}
			full step:\\
			add constraint $3$
		\end{tabular}\\
		8 & $\begin{bmatrix}
			2\\
			-1\\
			0\\
			3\\
			0\\
			-2
		\end{bmatrix}$ & $\begin{bmatrix}
			-2\\
			-1\\
			0\\
			-3\\
			-2
		\end{bmatrix}$ & $\frac{13}{2}$ & $\left\lbrace 3\right\rbrace$ & $\begin{bmatrix}
			0\\
			0\\
			-2\\
			0\\
			0
		\end{bmatrix}$ & & & & & & & \begin{tabular}{l}
			stop:\\
			all constraints satisfied
		\end{tabular}\\
		\bottomrule
	\end{longtable}

	Example 2:
	\begin{align}
		\left\lbrace
		\begin{aligned}
			\min_{x\in\mathbb{R}^{n_{x}}}\max_{y\in\mathbb{R}^{n_{y}}}\quad& f(x, y) = \frac{1}{2}x^{T}G_{11}x + x^{T}G_{12}y + \frac{1}{2}y^{T}G_{22}y + c_{x}^{T}x + c_{y}^{T}y,\\
			{\rm s.t.}\quad & s(x, y) = Ax + By + h \leq 0.
		\end{aligned}\right.
	\end{align}
	where
	\begin{align*}
		G_{11} & = \begin{bmatrix}
			-1 & -1\\
			-1 & 0
		\end{bmatrix}, \quad G_{12} = \begin{bmatrix}
			2 &	-1 & 0 & 0\\
			1 &	0 &	-1 & 1
		\end{bmatrix}, \quad G_{22} = \begin{bmatrix}
			-3 & 1 & 0 & 1\\
			1 & -1 & 1 & -1\\
			0 & 1 & -2 & 2\\
			1 & -1 & 2 & -4
		\end{bmatrix},\\
		A & = \begin{bmatrix}
			0 & -1\\
			1 & 0\\
			1 & 1\\
			1 & 0
		\end{bmatrix},\quad B = \begin{bmatrix}
			0 & 1 & -2 & 1\\
			2 & 0 & 0 & 1\\
			2 & -1 & 2 & 2\\
			2 & 1 & 0 & 2
		\end{bmatrix},\\
		c_{x}^{T} & \begin{bmatrix}
			3 & 1
		\end{bmatrix}, \quad c_{y}^{T} = \begin{bmatrix}
			1 & 5 & -6 & 7
		\end{bmatrix},\quad h^{T} = \begin{bmatrix}
			-7 & -5 & 0 & 9
		\end{bmatrix}.
	\end{align*}
	It is straightforward to verify that
	\begin{align*}
		DG^{-1}D^{T} =
		\begin{bmatrix}
			-3 & 3 & 8 & 6\\
			3 & -12 & -21 & -27\\
			8 & -21 & -41 & -45\\
			6 & -27 & -45 & -52
		\end{bmatrix} \npreceq 0,
	\end{align*}
	and this sample satisfies Assumption~\ref{assumption}. 
	We obtain
	\begin{align*}
	z_{0}^{T} = -c^{T}G^{-1} =
	\begin{bmatrix}
		1 & -2 & 11 & 26 & 17 & 6
	\end{bmatrix},
	\end{align*}
	which yields the initial S-pair $(z_{0}, \emptyset)$. 
	Table~\ref{solution path2} summarizes several principal solution paths starting from $z_{0}$ for Example~2. 
	Solution paths~1–2 indicate that this sample admits two distinct local minimax points, denoted by $z^{\ast}$ and $\bar{z}$. 
	Paths~1, 3, and 5 converge to $z^{\ast}$, while paths~2, 4, and 6 converge to $\bar{z}$. 
	Unlike paths~1–5, in path~6 the objective value increases in iteration~3, where $f(z_{2}) < f(z_{5})$. 
	For path~6, the sequence of S-pair iterations is
	\begin{align*}
	(z_{0}, \emptyset)
	\rightarrow (z_{1}, \{3\})
	\rightarrow (z_{2}, \{3, 1\})
	\rightarrow (\bar{z}, \{4\}).
	\end{align*}
	Although the objective function increases in iteration~3, we have $f(z_{2}) > f(\bar{z})$, which is consistent with Theorem~\ref{theorem >=}. 
	Overall, the objective function decreases in every S-pair iteration. 
	Moreover, identical S-pairs are obtained for the same active set, such as iteration~3 in path~3 and iteration~5 in path~5.
	
	\begin{longtable}{>{\centering\arraybackslash}p{1.2cm}  
			>{\centering\arraybackslash}p{1.2cm}  
			>{\centering\arraybackslash}p{1.2cm}    
			>{\centering\arraybackslash}p{0.3cm}  
			>{\centering\arraybackslash}p{0.8cm}  
			>{\centering\arraybackslash}p{1.1cm}  
			>{\centering\arraybackslash}p{0.3cm}  
			>{\centering\arraybackslash}p{1.1cm}  
			>{\centering\arraybackslash}p{0.8cm}  
			>{\centering\arraybackslash}p{0.3cm}  
			>{\centering\arraybackslash}p{0.3cm}  
			>{\centering\arraybackslash}p{0.3cm}  
			>{\raggedright\arraybackslash}p{3.6cm} 
		}
		\caption{Solution paths of the dual algorithm for Example2}\label{solution path2}\\
		\toprule
		Iteration & \multicolumn{5}{l}{At the start of iteration} &\multicolumn{6}{l}{Computed during iteration} & Remarks \\
		\cmidrule(lr){2-6} \cmidrule(lr){7-12}
		& $z$ & $s$ & $f$ & $\alpha$ & $u$ & $p$ & $d$ & $r$ & $t_{1}$ & $t_{2}$ & $k$ &  \\
		\midrule
		\multicolumn{13}{l}{Solution path 1: $z_{0} \rightarrow z^{\ast}$.}\\
		\midrule
		1 & $\begin{bmatrix}
			1\\
			-2\\
			11\\
			26\\
			17\\
			6
		\end{bmatrix}$ & $\begin{bmatrix}
			-7\\
			24\\
			41\\
			52
		\end{bmatrix}$ & $41$ & $\emptyset$ & $\begin{bmatrix}
			0\\
			0\\
			0\\
			0
		\end{bmatrix}$ & $2$ & $\begin{bmatrix}
			1\\
			0\\
			-5\\
			-12\\
			-9\\
			-3
		\end{bmatrix}$ & $-$ & $\infty$ & $2$ & $-$ & \begin{tabular}{l}
			full step:\\
			add constraint $2$
		\end{tabular}\\
		2 & $\begin{bmatrix}
			3\\ -2\\ 1\\ 2\\ -1\\ 0
		\end{bmatrix}$ & $\begin{bmatrix}
			-1\\ 0\\ -1\\ -2
		\end{bmatrix}$ & $17$ & $\left\lbrace 2\right\rbrace$ & $\begin{bmatrix}
			0\\ -2\\ 0\\ 0
		\end{bmatrix}$ & & & & & & & \begin{tabular}{l}
		stop:\\
		all constraints satisfied
		\end{tabular}\\
		\midrule
		\multicolumn{13}{l}{Solution path 2: $z_{0} \rightarrow \bar{z}$.}\\
		\midrule
		1 & $\begin{bmatrix}
			1\\
			-2\\
			11\\
			26\\
			17\\
			6
		\end{bmatrix}$ & $\begin{bmatrix}
			-7\\
			24\\
			41\\
			52
		\end{bmatrix}$ & $41$ & $\emptyset$ & $\begin{bmatrix}
			0\\
			0\\
			0\\
			0
		\end{bmatrix}$ & $4$ & $\begin{bmatrix}
			-1\\ -1\\ -10\\ -19\\ -15\\ -6
		\end{bmatrix}$ & $-$ & $\infty$ & $1$ & $-$ & \begin{tabular}{l}
			full step:\\
			add constraint $4$
		\end{tabular}\\
		2 & $\begin{bmatrix}
			0\\ -3\\ 1\\ 7\\ 2\\ 0
		\end{bmatrix}$ & $\begin{bmatrix}
			-1\\ -3\\ -4\\ 0
		\end{bmatrix}$ & $15$ & $\left\lbrace 4\right\rbrace$ & $\begin{bmatrix}
			0\\ 0\\ 0\\ -1
		\end{bmatrix}$ & & & & & & & \begin{tabular}{l}
		stop:\\
		all constraints satisfied
		\end{tabular}\\
		\midrule
		\multicolumn{13}{l}{Solution path 3: $z_{0} \rightarrow z_{1} \rightarrow z_{4} \rightarrow z^{\ast}$.}\\
		\midrule
		1 & $\begin{bmatrix}
			1\\
			-2\\
			11\\
			26\\
			17\\
			6
		\end{bmatrix}$ & $\begin{bmatrix}
			-7\\
			24\\
			41\\
			52
		\end{bmatrix}$ & $41$ & $\emptyset$ & $\begin{bmatrix}
			0\\
			0\\
			0\\
			0
		\end{bmatrix}$ & $3$ & $\begin{bmatrix}
			1\\ 0\\ -8\\ -18\\ -16\\ -6
		\end{bmatrix}$ & $-$ & $\infty$ & $1$ & $-$ & \begin{tabular}{l}
			full step:\\
			add constraint $3$
		\end{tabular}\\
		2 & $\begin{bmatrix}
			2\\ -2\\ 3\\ 8\\ 1\\ 0
		\end{bmatrix}$ & $\begin{bmatrix}
			1\\ 3\\ 0\\ 7
		\end{bmatrix}$ & $\frac{41}{2}$ & $\left\lbrace 3\right\rbrace$ & $\begin{bmatrix}
			0\\ 0\\ -1\\ 0
		\end{bmatrix}$ & $2$ & $\begin{bmatrix}
			\frac{20}{41}\\ 0\\ -\frac{37}{41}\\ -\frac{114}{41}\\ -\frac{33}{41}\\ \frac{3}{41}
		\end{bmatrix}$ & $\begin{bmatrix}
			-\frac{21}{41}
		\end{bmatrix}$ & $\frac{41}{21}$ & $\frac{41}{17}$ & $3$ & \begin{tabular}{l}
			partial step:\\
			drop constraint $3$
		\end{tabular}\\
		3 & $\begin{bmatrix}
			\frac{62}{21}\\ -2\\ \frac{26}{21}\\ \frac{18}{7}\\ -\frac{4}{7}\\ \frac{1}{7}
		\end{bmatrix}$ & $\begin{bmatrix}
			-\frac{8}{7}\\ \frac{4}{7}\\ 0\\ -\frac{5}{7}
		\end{bmatrix}$ & $\frac{2665}{147}$ & $\emptyset$ & $\begin{bmatrix}
			0\\ -\frac{41}{21}\\ 0\\ 0
		\end{bmatrix}$ & $2$ & $\begin{bmatrix}
			1\\ 0\\ -5\\ -12\\ -9\\ -3
		\end{bmatrix}$ & $-$ & $\infty$ & $\frac{1}{21}$ & $-$ & \begin{tabular}{l}
			full step:\\
			add constraint $2$
		\end{tabular}\\
		4 & $\begin{bmatrix}
			3\\ -2\\ 1\\ 2\\ -1\\ 0
		\end{bmatrix}$ & $\begin{bmatrix}
			-1\\ 0\\ -1\\ -2
		\end{bmatrix}$ & $17$ & $\left\lbrace 2\right\rbrace$ & $\begin{bmatrix}
			0\\ -2\\ 0\\ 0
		\end{bmatrix}$ & & & & & & & \begin{tabular}{l}
			stop:\\
			all constraints satisfied
		\end{tabular}\\
		\midrule
		\multicolumn{13}{l}{Solution path 4: $z_{0} \rightarrow z_{1} \rightarrow z_{6} \rightarrow \bar{z}$.}\\
		\midrule
		1 & $\begin{bmatrix}
			1\\
			-2\\
			11\\
			26\\
			17\\
			6
		\end{bmatrix}$ & $\begin{bmatrix}
			-7\\
			24\\
			41\\
			52
		\end{bmatrix}$ & $41$ & $\emptyset$ & $\begin{bmatrix}
			0\\
			0\\
			0\\
			0
		\end{bmatrix}$ & $3$ & $\begin{bmatrix}
			1\\ 0\\ -8\\ -18\\ -16\\ -6
		\end{bmatrix}$ & $-$ & $\infty$ & $1$ & $-$ & \begin{tabular}{l}
			full step:\\
			add constraint $3$
		\end{tabular}\\
		2 & $\begin{bmatrix}
			2\\ -2\\ 3\\ 8\\ 1\\ 0
		\end{bmatrix}$ & $\begin{bmatrix}
			1\\ 3\\ 0\\ 7
		\end{bmatrix}$ & $\frac{41}{2}$ & $\left\lbrace 3\right\rbrace$ & $\begin{bmatrix}
			0\\ 0\\ -1\\ 0
		\end{bmatrix}$ & $4$ & $\begin{bmatrix}
			-\frac{86}{41}\\ -1\\ -\frac{50}{41}\\ \frac{31}{41}\\ \frac{105}{41}\\ \frac{24}{41}
		\end{bmatrix}$ & $\begin{bmatrix}
			-\frac{45}{41}
		\end{bmatrix}$ & $\frac{41}{45}$ & $\frac{287}{107}$ & $3$ & \begin{tabular}{l}
			partial step:\\
			drop constraint $3$
		\end{tabular}\\
		3 & $\begin{bmatrix}
			\frac{4}{45}\\ -\frac{131}{45}\\ \frac{17}{9}\\ \frac{391}{45}\\ \frac{10}{3}\\ \frac{8}{15}
		\end{bmatrix}$ & $\begin{bmatrix}
			-\frac{23}{15}\\ -\frac{3}{5}\\ 0\\ \frac{208}{45}
		\end{bmatrix}$ & $\frac{13184}{679}$ & $\emptyset$ & $\begin{bmatrix}
			0\\ 0\\ 0\\ -\frac{41}{45}
		\end{bmatrix}$ & $4$ & $\begin{bmatrix}
			-1\\ -1\\ -10\\ -19\\ -15\\ -6
		\end{bmatrix}$ & $-$ & $\infty$ & $\frac{4}{45}$ & $-$ & \begin{tabular}{l}
			full step:\\
			add constraint $4$
		\end{tabular}\\
		4 & $\begin{bmatrix}
			0\\ -3\\ 1\\ 7\\ 2\\ 0
		\end{bmatrix}$ & $\begin{bmatrix}
			-1\\ -3\\ -4\\ 0
		\end{bmatrix}$ & $15$ & $\left\lbrace 4\right\rbrace$ & $\begin{bmatrix}
			0\\ 0\\ 0\\ -1
		\end{bmatrix}$ & & & & & & & \begin{tabular}{l}
			stop:\\
			all constraints satisfied
		\end{tabular}\\
		\midrule
		\multicolumn{13}{l}{Solution path 5: $z_{0} \rightarrow z_{1} \rightarrow z_{2} \rightarrow z_{3} \rightarrow z_{4} \rightarrow z^{\ast}$.}\\
		\midrule
		1 & $\begin{bmatrix}
			1\\
			-2\\
			11\\
			26\\
			17\\
			6
		\end{bmatrix}$ & $\begin{bmatrix}
			-7\\
			24\\
			41\\
			52
		\end{bmatrix}$ & $41$ & $\emptyset$ & $\begin{bmatrix}
			0\\
			0\\
			0\\
			0
		\end{bmatrix}$ & $3$ & $\begin{bmatrix}
			1\\ 0\\ -8\\ -18\\ -16\\ -6
		\end{bmatrix}$ & $-$ & $\infty$ & $1$ & $-$ & \begin{tabular}{l}
			full step:\\
			add constraint $3$
		\end{tabular}\\
		2 & $\begin{bmatrix}
			2\\ -2\\ 3\\ 8\\ 1\\ 0
		\end{bmatrix}$ & $\begin{bmatrix}
			1\\ 3\\ 0\\ 7
		\end{bmatrix}$ & $\frac{41}{2}$ & $\left\lbrace 3\right\rbrace$ & $\begin{bmatrix}
			0\\ 0\\ -1\\ 0
		\end{bmatrix}$ & $1$ & $\begin{bmatrix}
			\frac{8}{41}\\ 0\\ -\frac{23}{41}\\ -\frac{62}{41}\\ -\frac{5}{41}\\ -\frac{7}{41}
		\end{bmatrix}$ & $\begin{bmatrix}
			\frac{8}{41}
		\end{bmatrix}$ & $\infty$ & $\frac{59}{41}$ & $-$ & \begin{tabular}{l}
			full step:\\
			add constraint $1$
		\end{tabular}\\
		3 & $\begin{bmatrix}
			\frac{126}{59}\\ -2\\ \frac{154}{59}\\ \frac{410}{59}\\ \frac{54}{59}\\ -\frac{7}{59}
		\end{bmatrix}$ & $\begin{bmatrix}
			0\\ \frac{132}{59}\\ 0\\ \frac{299}{59}
		\end{bmatrix}$ & $\frac{1189}{59}$ & $\left\lbrace 3, 1\right\rbrace$ & $\begin{bmatrix}
			-\frac{41}{59}\\ 0\\ -\frac{67}{59}\\ 0
		\end{bmatrix}$ & $2$ & $\begin{bmatrix}
			\frac{20}{59}\\ 0\\ -\frac{28}{59}\\ -\frac{96}{59}\\ -\frac{42}{59}\\ \frac{12}{59}
		\end{bmatrix}$ & $\begin{bmatrix}
			-\frac{39}{59}\\ -\frac{45}{59}
		\end{bmatrix}$ & $\frac{41}{45}$ & $\frac{11}{2}$ & $1$ & \begin{tabular}{l}
			partial step:\\
			drop constraint $1$
		\end{tabular}\\
		4 & $\begin{bmatrix}
				\frac{22}{9}\\ -2\\ \frac{98}{45}\\ \frac{82}{15}\\ \frac{4}{15}\\ \frac{1}{15}
		\end{bmatrix}$ & $\begin{bmatrix}
			0\\ \frac{28}{15}\\ 0\\ \frac{17}{5}
		\end{bmatrix}$ & $\frac{3677}{184}$ & $\left\lbrace 3\right\rbrace$ & $\begin{bmatrix}
			0\\ -\frac{41}{45}\\ -\frac{8}{45}\\ 0
		\end{bmatrix}$ & $2$ & $\begin{bmatrix}
			\frac{20}{41}\\ 0\\ -\frac{37}{41}\\ -\frac{114}{41}\\ -\frac{33}{41}\\ \frac{3}{41}
		\end{bmatrix}$ & $\begin{bmatrix}
			-\frac{21}{41}
		\end{bmatrix}$ & $\frac{328}{315}$ & $\frac{1148}{765}$ & $3$ & \begin{tabular}{l}
			partial step:\\
			drop constraint $3$
		\end{tabular}\\
		5 & $\begin{bmatrix}
				\frac{62}{21}\\ -2\\ \frac{26}{21}\\ \frac{18}{7}\\ -\frac{4}{7}\\ \frac{1}{7}
		\end{bmatrix}$ & $\begin{bmatrix}
			-\frac{8}{7}\\ \frac{4}{7}\\ 0\\ -\frac{5}{7}
		\end{bmatrix}$ & $\frac{2665}{147}$ & $\emptyset$ & $\begin{bmatrix}
			0\\ -\frac{41}{21}\\ 0\\ 0
		\end{bmatrix}$ & $2$ & $\begin{bmatrix}
			1\\ 0\\ -5\\ -12\\ -9\\ -3
		\end{bmatrix}$ & $-$ & $\infty$ & $\frac{1}{21}$ & $-$ & \begin{tabular}{l}
		full step:\\
		add constraint $2$
		\end{tabular}\\
		6 & $\begin{bmatrix}
		3\\ -2\\ 1\\ 2\\ -1\\ 0
		\end{bmatrix}$ & $\begin{bmatrix}
		-1\\ 0\\ -1\\ -2
		\end{bmatrix}$ & $17$ & $\left\lbrace 2\right\rbrace$ & $\begin{bmatrix}
		0\\ -2\\ 0\\ 0
		\end{bmatrix}$ & & & & & & & \begin{tabular}{l}
		stop:\\
		all constraints satisfied
		\end{tabular}\\
		\midrule
		\multicolumn{13}{l}{Solution path 6: $z_{0} \rightarrow z_{1} \rightarrow z_{2} \rightarrow z_{5} \rightarrow z_{6} \rightarrow \bar{z}$.}\\
		\midrule
		1 & $\begin{bmatrix}
			1\\
			-2\\
			11\\
			26\\
			17\\
			6
		\end{bmatrix}$ & $\begin{bmatrix}
			-7\\
			24\\
			41\\
			52
		\end{bmatrix}$ & $41$ & $\emptyset$ & $\begin{bmatrix}
			0\\
			0\\
			0\\
			0
		\end{bmatrix}$ & $3$ & $\begin{bmatrix}
			1\\ 0\\ -8\\ -18\\ -16\\ -6
		\end{bmatrix}$ & $-$ & $\infty$ & $1$ & $-$ & \begin{tabular}{l}
			full step:\\
			add constraint $3$
		\end{tabular}\\
		2 & $\begin{bmatrix}
			2\\ -2\\ 3\\ 8\\ 1\\ 0
		\end{bmatrix}$ & $\begin{bmatrix}
			1\\ 3\\ 0\\ 7
		\end{bmatrix}$ & $\frac{41}{2}$ & $\left\lbrace 3\right\rbrace$ & $\begin{bmatrix}
			0\\ 0\\ -1\\ 0
		\end{bmatrix}$ & $1$ & $\begin{bmatrix}
			\frac{8}{41}\\ 0\\ -\frac{23}{41}\\ -\frac{62}{41}\\ -\frac{5}{41}\\ -\frac{7}{41}
		\end{bmatrix}$ & $\begin{bmatrix}
			\frac{8}{41}
		\end{bmatrix}$ & $\infty$ & $\frac{59}{41}$ & $-$ & \begin{tabular}{l}
			full step:\\
			add constraint $1$
		\end{tabular}\\
		3 & $\begin{bmatrix}
				\frac{126}{59}\\ -2\\ \frac{154}{59}\\ \frac{410}{59}\\ \frac{54}{59}\\ -\frac{7}{59}
		\end{bmatrix}$ & $\begin{bmatrix}
			0\\ \frac{132}{59}\\ 0\\ \frac{299}{59}
		\end{bmatrix}$ & $\frac{1189}{59}$ & $\left\lbrace 3, 1\right\rbrace$ & $\begin{bmatrix}
			-\frac{41}{59}\\ 0\\ -\frac{67}{59}\\ 0
		\end{bmatrix}$ & $4$ & $\begin{bmatrix}
			-\frac{146}{59}\\ -1\\ -\frac{8}{59}\\ \frac{217}{59}\\ \frac{165}{59}\\ \frac{54}{59}
		\end{bmatrix}$ & $\begin{bmatrix}
			-\frac{87}{59}\\ -\frac{114}{59}
		\end{bmatrix}$ & $\frac{41}{114}$ & $\infty$ & $1$ & \begin{tabular}{l}
			partial step:\\
			drop constraint $1$
		\end{tabular}\\
		4 & $\begin{bmatrix}
				\frac{71}{57}\\ -\frac{269}{114}\\ \frac{146}{57}\\ \frac{943}{114}\\ \frac{73}{38}\\ \frac{4}{19}
		\end{bmatrix}$ & $\begin{bmatrix}
			0\\ \frac{30}{19}\\ 0\\ \frac{691}{114}
		\end{bmatrix}$ & $\frac{3192}{157}$ & $\left\lbrace 3\right\rbrace$ & $\begin{bmatrix}
			0\\ 0\\ -\frac{23}{38}\\ -\frac{41}{114}
		\end{bmatrix}$ & $4$ & $\begin{bmatrix}
			-\frac{86}{41}\\ -1\\ -\frac{50}{41}\\ \frac{31}{41}\\ \frac{105}{41}\\ \frac{24}{41}
		\end{bmatrix}$ & $\begin{bmatrix}
			-\frac{45}{41}
		\end{bmatrix}$ & $\frac{943}{1710}$ & $\frac{5695}{2452}$ & $3$ & \begin{tabular}{l}
			partial step:\\
			drop constraint $3$
		\end{tabular}\\
		5 & $\begin{bmatrix}
				\frac{4}{45}\\ -\frac{131}{45}\\ \frac{17}{9}\\ \frac{391}{45}\\ \frac{10}{3}\\ \frac{8}{15}
		\end{bmatrix}$ & $\begin{bmatrix}
			-\frac{23}{15}\\ -\frac{3}{5}\\ 0\\ \frac{208}{45}
		\end{bmatrix}$ & $\frac{13184}{679}$ & $\emptyset$ & $\begin{bmatrix}
			0\\ 0\\ 0\\ -\frac{41}{45}
		\end{bmatrix}$ & $4$ & $\begin{bmatrix}
			-1\\ -1\\ -10\\ -19\\ -15\\ -6
		\end{bmatrix}$ & $-$ & $\infty$ & $\frac{4}{45}$ & $-$ & \begin{tabular}{l}
			full step:\\
			add constraint $4$
		\end{tabular}\\
		6 & $\begin{bmatrix}
			0\\ -3\\ 1\\ 7\\ 2\\ 0
		\end{bmatrix}$ & $\begin{bmatrix}
			-1\\ -3\\ -4\\ 0
		\end{bmatrix}$ & $15$ & $\left\lbrace 4\right\rbrace$ & $\begin{bmatrix}
			0\\ 0\\ 0\\ -1
		\end{bmatrix}$ & & & & & & & \begin{tabular}{l}
			stop:\\
			all constraints satisfied
		\end{tabular}\\
		\bottomrule
	\end{longtable}
	
	\section{Conclusion}\label{section7}
	
	In this paper, we investigate a minimax quadratic programming problem with coupled inequality constraints. 
	Motivated by the augmented Lagrangian method for equality constrained minimax optimization problems proposed by Dai and Zhang~\cite{daiRateConvergenceAugmented2024}, 
	we focus on extending dual active set methods to inequality constrained settings by leveraging the duality theorem introduced by Tsaknakis et al.~\cite{tsaknakisMinimaxProblemsCoupled2023}. 
	Under Assumption~\ref{assumption}, we establish that the S-pair does not repeat and that the proposed dual algorithm terminates in a finite number of iterations, 
	ensured by the monotonic decrease of the objective function value. 
	We further develop a numerically stable implementation of the algorithm using Cholesky factorization and Givens rotations, 
	and validate its performance through numerical experiments including randomly generated minimax quadratic programs and an adversarial attack on a mean-covariance portfolio model. Additionally, we provide illustrative examples that detail the dual algorithm’s iterative behavior.
	
	Several directions for future research remain open. 
	First, although the proposed method ensures finite termination under Assumption~\ref{assumption}, 
	an effective strategy for selecting the violated constraint may substantially enhance computational efficiency and applicability to broader problem classes. 
	Second, it is of interest to extend this method to sequential minimax quadratic programming, 
	drawing inspiration from sequential quadratic programming (SQP) methods in nonlinear optimization.
	
	\bibliographystyle{unsrt}
	
\end{document}